\documentclass[11 pt]{amsart}
\usepackage{latexsym,amsmath,amsfonts,graphicx}
\usepackage{graphicx}
\usepackage{subfigure}
\usepackage{amssymb,cases}
\usepackage{amstext}
\usepackage[dvips]{color}
\newtheorem{theorem}{Theorem}[section]
\newtheorem{thm}[theorem]{Theorem}
\newtheorem{pro}{Proposition}[section]
\newtheorem{cor}[pro]{Corollary}
\newtheorem{lemma}[pro]{Lemma}
\newtheorem{remark}[pro]{Remark}
\newtheorem{defi}[pro]{Definition}

\newtheorem{example}[pro]{Example}
\numberwithin{equation}{section}

\def\cal{\mathcal }
\def\R{\mathbb R}

\def\Z{\mathbb Z}

\def\L{\mathbb L}
\def\mathscr{\mathcal }

\newcommand{\bomega}{{\boldsymbol {\omega}}}
\newcommand{\bgamma}{{\boldsymbol {\gamma}}}
\newcommand{\bdeta}{{\boldsymbol {\eta}}}

\newcommand{\mi}{{\mathbf{i}}}

\begin{document}

\title[Optimal parameterizations of self-similar sets]
{
Space-filling curves of self-similar sets (I): Iterated function systems with order structures}

\author{Hui Rao}
\address{Hui Rao: Department of Mathematics, Hua Zhong Normal University}
\email{hrao@mail.ccnu.edu.cn}
\author{Shu-Qin Zhang$\dag$}
\address{Shu-Qin Zhang: Department of Mathematics, Hua Zhong Normal University}
\email{zhangsq\_ccnu@sina.com}
\thanks{$\dag$ The correspondence author.}
\thanks{The work is supported by CNFS Nos 11431007, 11171128 and 11471075.}

\begin{abstract}
 This paper is the first paper of three papers in a series, which intend to 
   provide a  systematic   treatment for the   space-filling curves of self-similar sets.

In the present paper,  we introduce a notion of \emph{linear  graph-directed IFS} (linear GIFS in short).
We show that to construct a space-filling curve of a self-similar set, it is amount to explore
its linear GIFS structures.  Some other notions, such as chain condition, path-on-lattice IFS, and visualizations of
space-filling curves are also concerned.

In sequential papers \cite{Dai15} and \cite{RZ14}, we obtain a universal algorithm to construct space-filling curves of self-similar sets of finite type, that is, as soon as the IFS is given, the computer will do everything automatically.
  Our study extends almost all the known results on space-filling curves.

  \medskip

\end{abstract}

\maketitle

\section{\textbf{Introduction}}
Space-filling curves have fascinated mathematicians for over a century.
Its history started with the monumental result of Peano in 1890 (\cite{Peano1890}).
 One year later, Hilbert
gave an alternative construction, now called Hilbert curve. In 1921,  Sierpi\'nski
discovered Sierpi\'nski space-filling curve, and it was generalized by  P\'olya. (See \cite{Ciesie2012}.)
For variations of the above constructions, see the survey book of Sagan \cite{Hans94}.

 \begin{figure}[h]
 \includegraphics[width=.32 \textwidth]{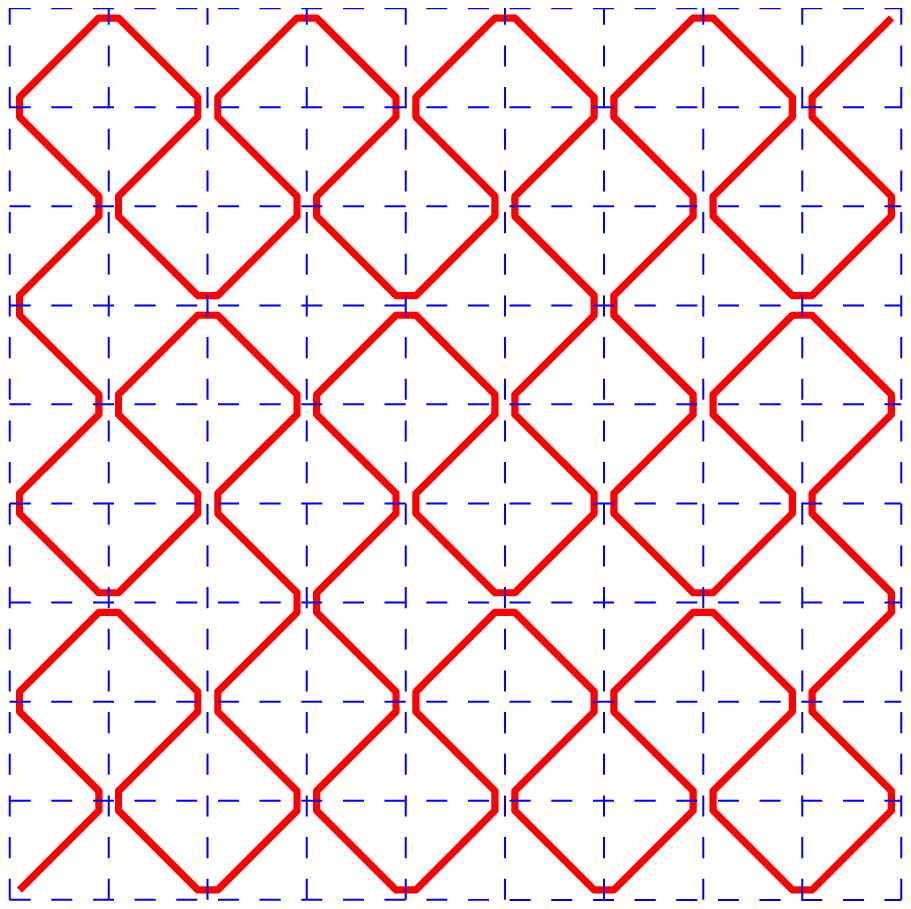}
 \includegraphics[width=.32 \textwidth]{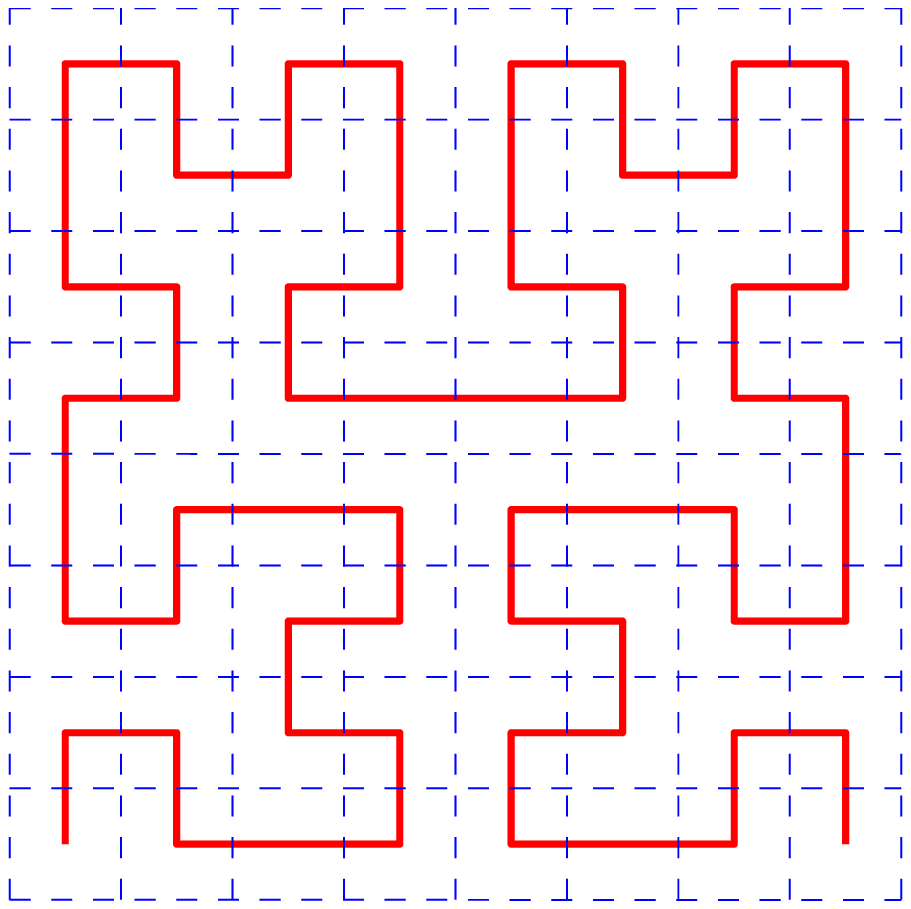}
  \includegraphics[width=.32 \textwidth]{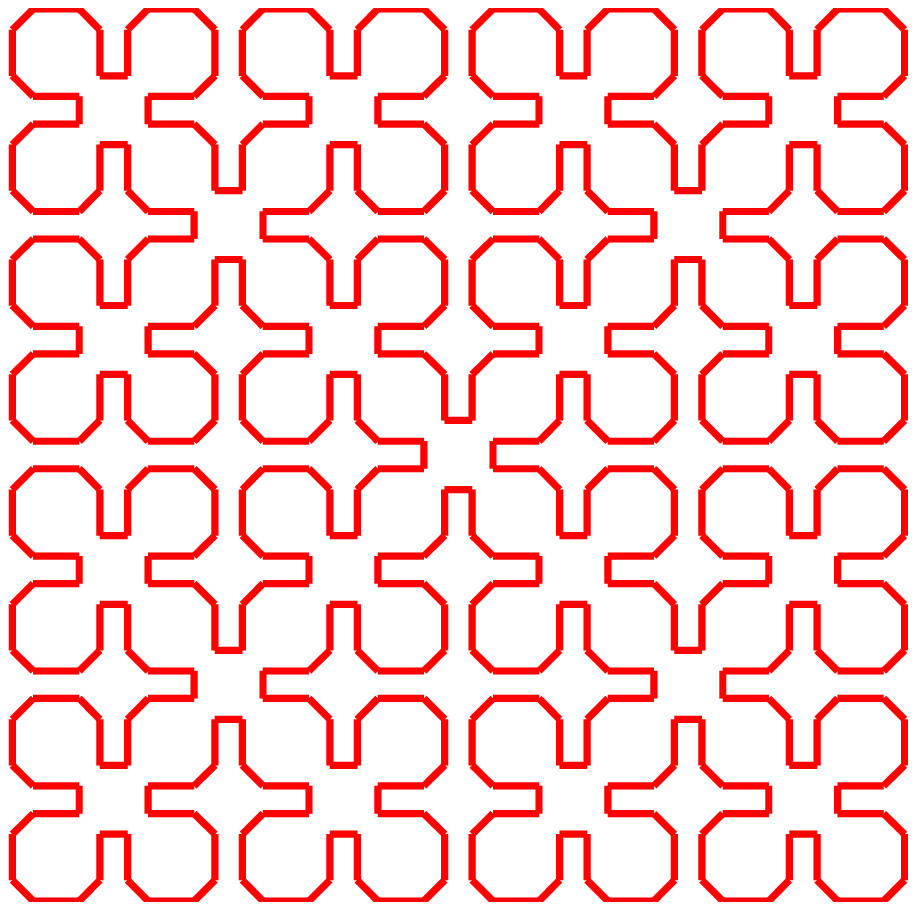}
  \caption{Space-filling curves of Peano, Hilbert and Sierpi\'nski.}\label{PHS}
\end{figure}

Around 1970's, several remarkable progresses have been made:
  J. Heighway, a physicist, found the Heighway dragon (\cite{Gardner76, Davis70});
W. Gosper, a computer scientist,  found the Gosper island (\cite{Gardner76});
H. Lindenmayer, a biologist,  introduced L-system (\cite{Lind68}), which becomes a powerful method to produce space-filling curves later.

\begin{figure}[h]
  \includegraphics[width=0.35 \textwidth]{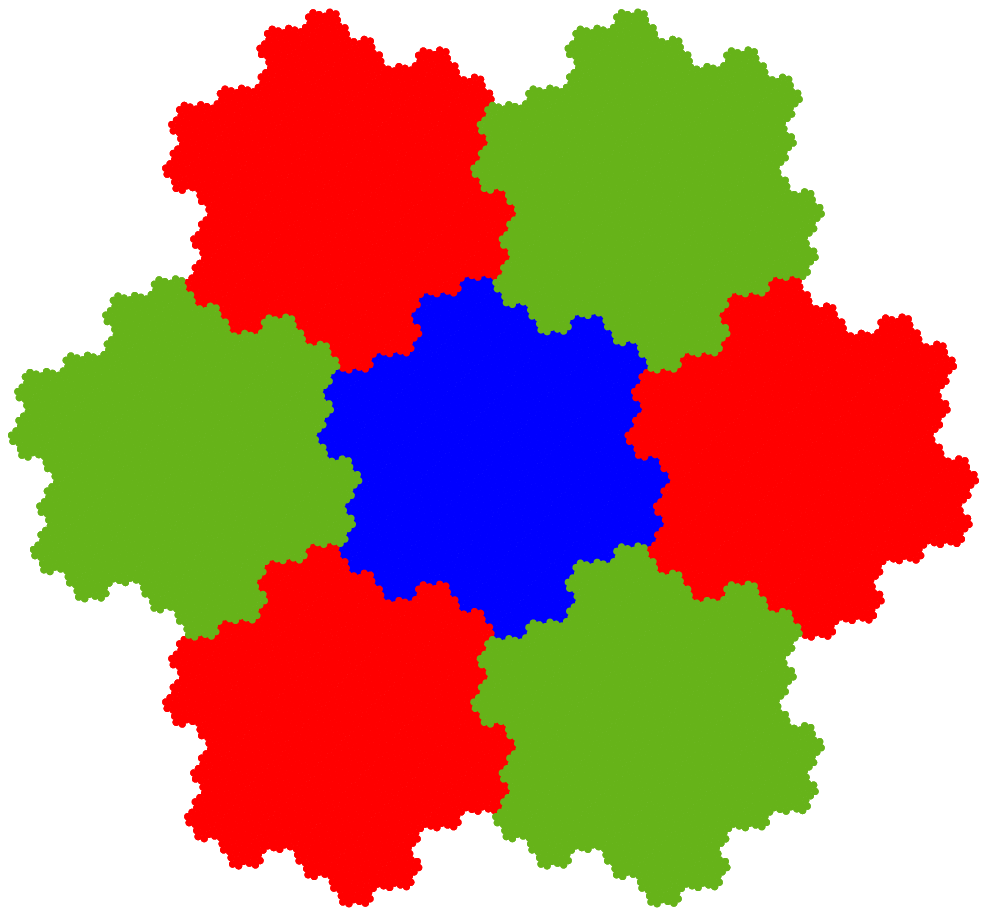} \hspace{-1.cm}
  \includegraphics[width=0.35 \textwidth]{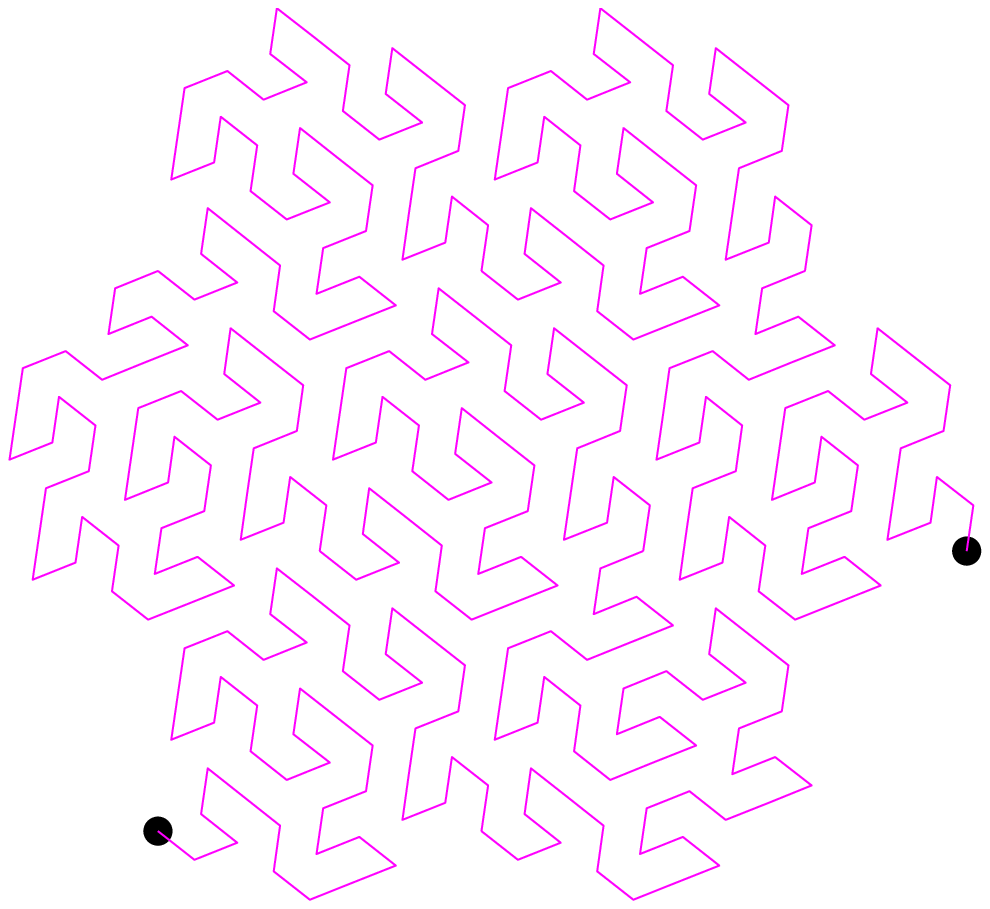} \hspace{-1.cm}
  \caption{The Gosper island is a 7-reptile, and the Gosper curve is an optimal parametrization.}\label{pic-Gosper}
\end{figure}

Two important facts are gradually recognized: All the constructions are based on certain self-similar structures,
and certain `substitution rules' play an essential role in the constructions.

Next major progress was made by
Dekking \cite{Dekking82} (1982),
  where he claimed that  he ``introduce a powerful method of describing and generating space-filling curves''.
This paper has important impact on both space-filling curves and fractal geometry.
On the fractal geometry aspect, \cite{Dekking82} leads to the emerge of the notion of graph-directed iterated function system. On the space-filling curve aspect, Dekking's method has been accepted by  computer scientists and as the
``vector method''.

In recent years, various interesting constructions of space-filling curves appear  on the internet,
for example, ``www.fractalcurves.com'' (see \cite{Ventrella})
     and  ``teachout1.net/village/'' (see \cite{Gary}). 
 Besides, space-filling curves of
higher dimensional cubes have been studied  by
Milne \cite{Milne80} and  Gilbert \cite{Gilbert84}.
For applications of space-filling curves,
see Bader \cite{Bader13} and the references therein.


 In this paper and two sequential papers, we unveil the mystery of space-filling curves by providing a rigorous and systematic treatment.

First, let us specify our meaning of space-filling curves.
We call  an onto mapping from an interval $[a,b]$ to a self-similar set $K$ \emph{an optimal parametrization},
   if it is almost one-to-one, measure-preserving and $1/s$-H\"older
 continuous, where $s=\dim_H K$ is the Hausdorff dimension of $K$. (For precise definition, see Section \ref{sec-basic}.)
 It is  observed that   most classical space-filling curves
fulfill the above requirements (\cite{Milne80, Gilbert84}), while some others like the Lebesgue curve does not (see Figure \ref{fig-Leb}(right)).
 It is proper to call an optimal parametrization   a \emph{space-filling curve}  if $K$ has non-empty interior,
and  call it  a \emph{fractal-filling curve} otherwise. However, for simplicity, we shall just call an optimal parametrization a space-filling curve.
 
\begin{figure}
  \includegraphics[width=.35 \textwidth]{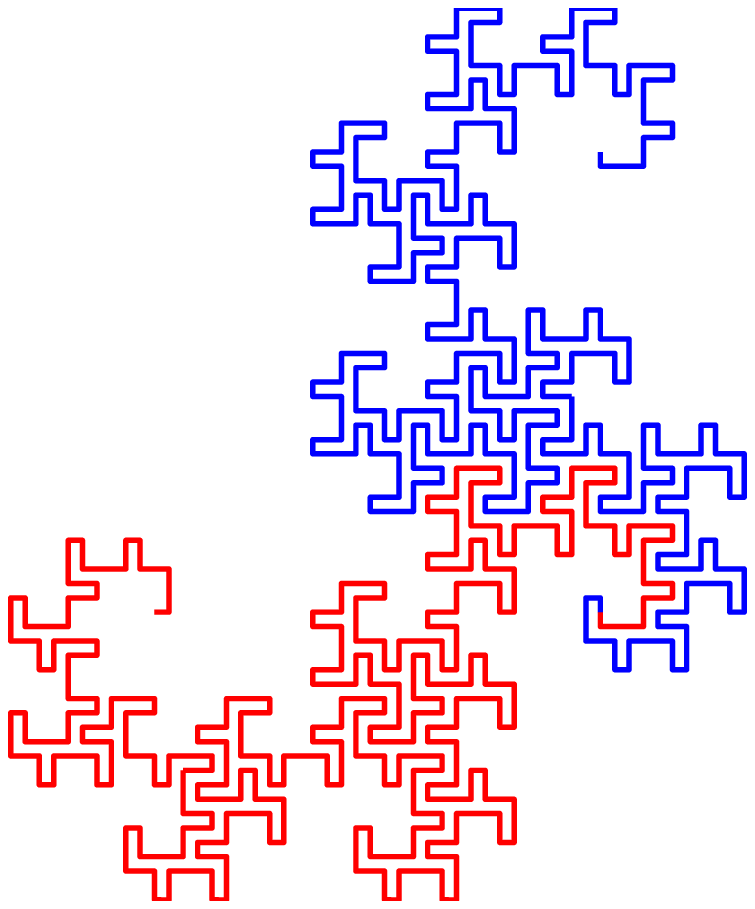}
  \includegraphics[width=.33 \textwidth]{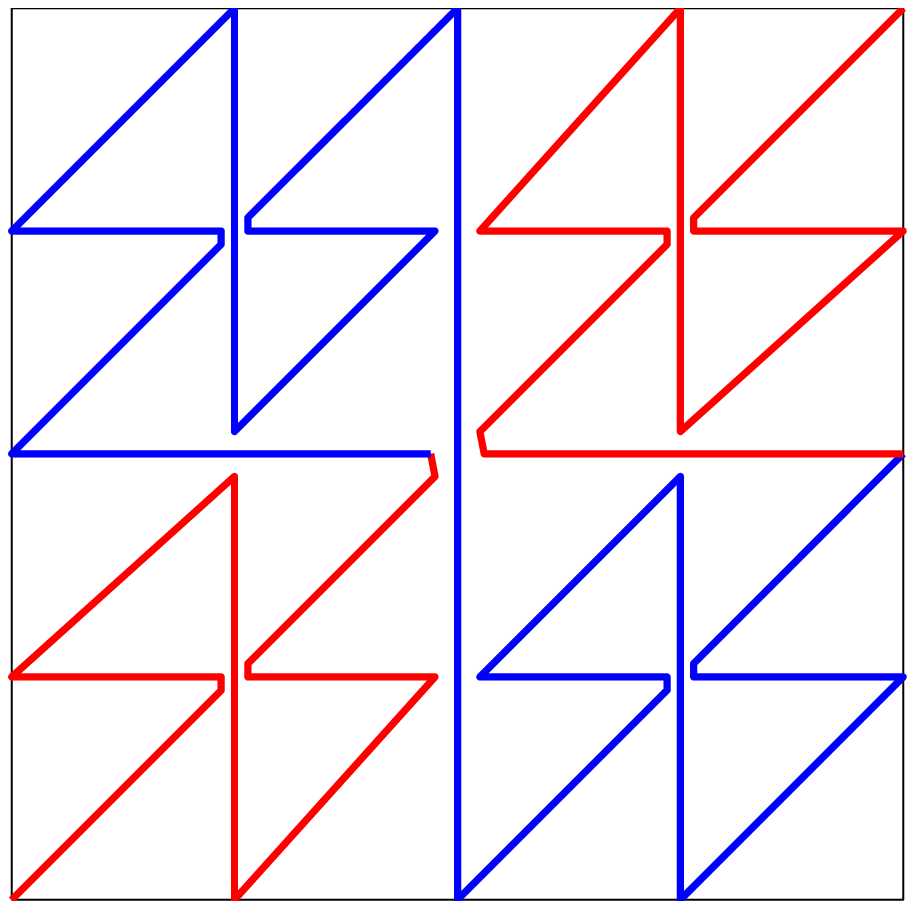}\\
  \caption{Heighway dragon curve and Lebesgue curve.}\label{fig-Leb}
\end{figure}

The main contribution of this paper is that  we introduce a notion of \emph{linear GIFS} to describe 
and handle space-filling curves.
 The \emph{graph-directed iterated function system}, or GIFS in short,
 is an important notion in fractal geometry.
 We equip the functions in a GIFS with a partial order and call it an ordered GIFS, and
  this order  induces a dictionary order of the associated symbolic space. An ordered GIFS is called a linear GIFS, if every two consecutive cylinders have non-empty intersections (see Section \ref{sec-GIFS} for precise definition). We show that

\begin{theorem}\label{M-theo1}
Let  $\{E_j\}_{j=1}^N$
be the invariant sets
 of a linear graph-directed IFS satisfying the open set condition and $0<\mathcal{H}^\delta(E_j)<\infty$ for $j=1, \dots, N$, where $\delta$ is the similarity dimension, then $E_j$ admits  optimal parametrizations for every $j=1,\dots, N$.
\end{theorem}

The proof of Theorem \ref{M-theo1} is constructive;
hence, to construct space-filling curves  is amount to seek a linear GIFS structure of the given set.
The common point of the L-language method and Dekking's vector method is that, first they construct
a linear GIFS, and then verify the open set condition.

\begin{remark}{\rm $(i)$ The notion of linear GIFS  can be regarded as a completion of the study of Dekking \cite{Dekking82}.

$(ii)$ After we finished this paper, we acknowledge that an idea similar to our linear GIFS has appeared in Akiyama and Loridant \cite{Akiyama10, Akiyama11} when studying the parameterizations of boundaries of self-affine tiles.
}
\end{remark}

For an ordered GIFS, one can associate to each invariant set $E_j$ a head (the point with the lowest coding)
and a tail  (the point with the highest coding).  Using heads and tails, we define a chain condition (see Section \ref{sec-chain})
 which
provides a simple and practical criterion of linear GIFS.

\begin{thm}\label{theo-chain} An ordered GIFS is a linear GIFS if and only if it satisfies the chain condition.
\end{thm}

To `see' a space-filling curve,  we need to visualize or to approximate a space-filling curve.
Using linear GIFS, in Section \ref{sec-visual}, we give a precise definition of visualizations  of a space-filling curve.

To illustrate our theory, we give a brief introduction to the path-on-lattice IFS in Section \ref{sec-path}.
A nice collection of space-filling curves given by path-on-lattice IFS,  many of them are well-known, can be
found in the website \cite{Ventrella}.  A detailed study of the path-on-lattice IFS
can be found in \cite{Yang16}.

To find the linear GIFS structure of a given self-similar set is a hard question.
This question is studied in  sequential papers \cite{Dai15} and \cite{RZ14}.
We show that

\begin{thm} (\cite{Dai15} and \cite{RZ14})
 Let $K$ be a connected self-similar set satisfying the open set condition.
 if $K$ has the finite skeleton property, then it admits  optimal parametrizations. In particular,  if $K$  satisfies  a finite type condition (another important condition in
fractal geometry), then it possesses  finite skeletons and hence admits optimal parameterizations.
 \end{thm}

 Our theory gives a universal algorithm to find space-filling curves of self-similar set of finite type, that is, as soon as the IFS is given, the computer will do everything.
  Our study extends almost all the known results on space-filling curves, and shows the internal relation
  between the space-filling curve and the recent developments of fractal geometry.
 
\begin{figure}[h]
      \includegraphics[width=0.45\textwidth]{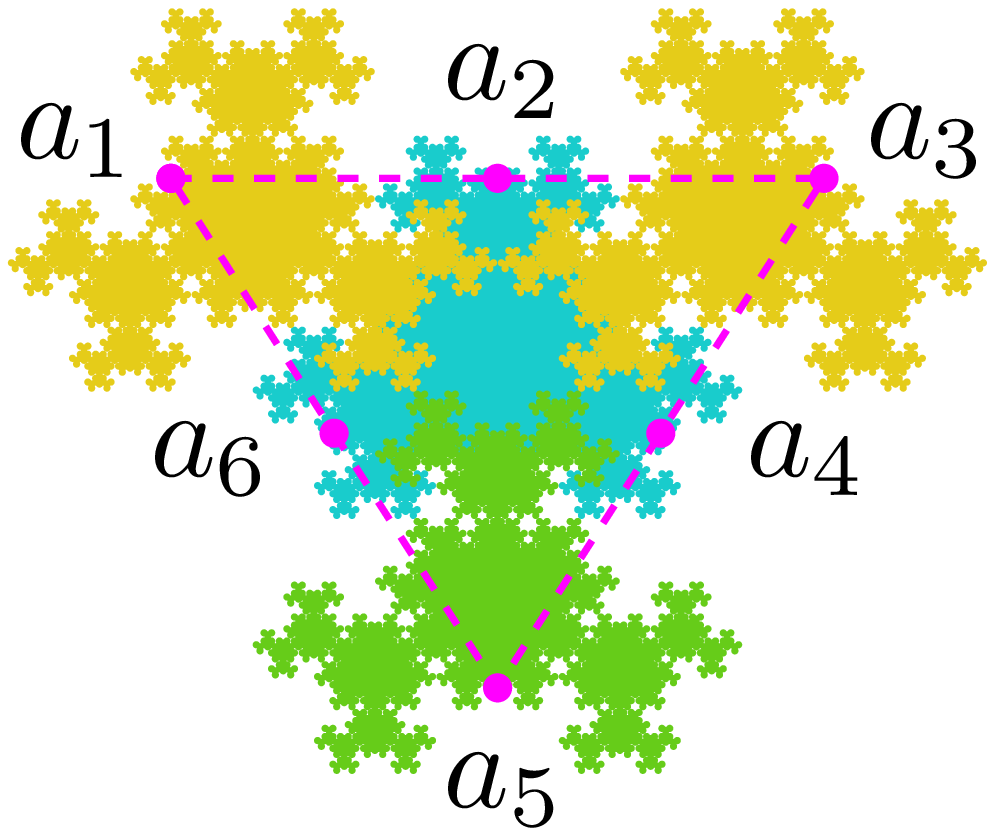}
    \includegraphics[width=0.4\textwidth]{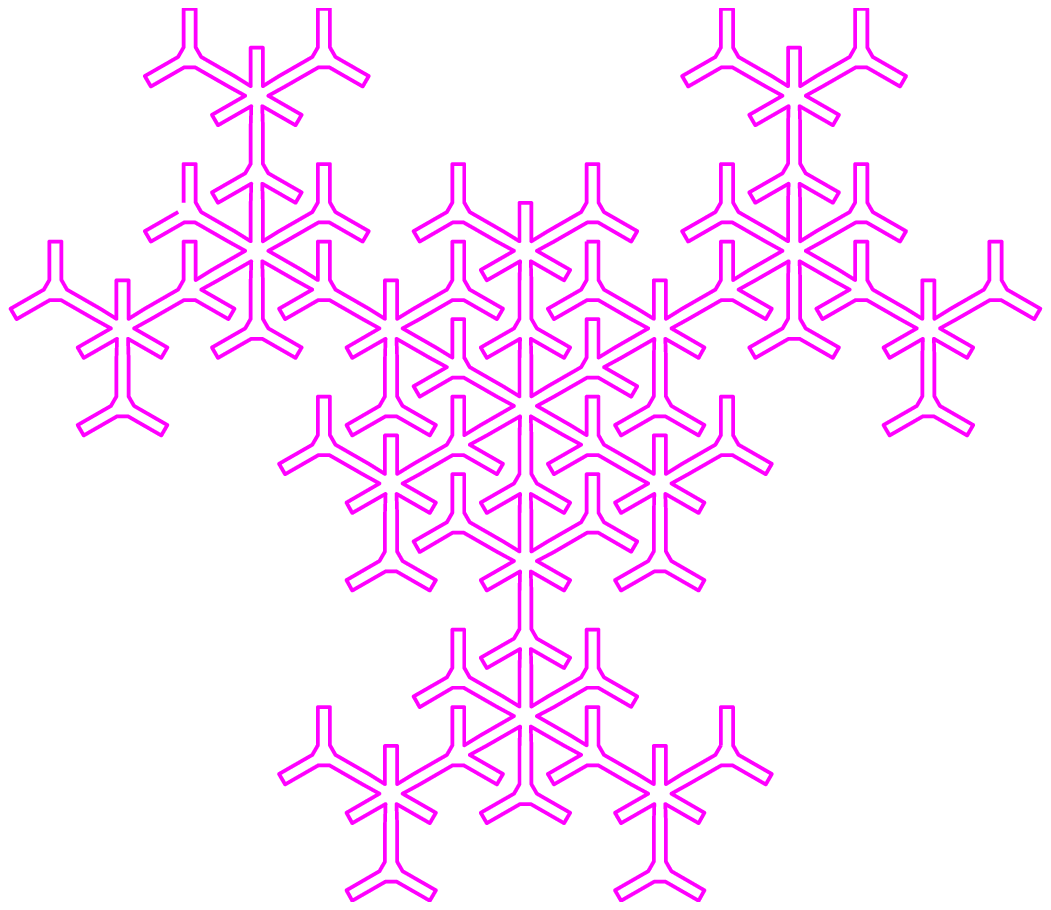}
  \caption{ The  four-star tile and a space-filling curve. }
  \label{4star}
\end{figure}

\begin{example}{ The four-star tile. \rm
Pictures in Figure \ref{4star} are taking from \cite{Gary}, but there is no explanation how to obtain the space-filling curve.
 Our study will fill all the gaps from the left picture to the right in Figure \ref{4star}, which is interesting and highly non-trivial (\cite{Dai15}).
 A sketch of the approach is provided in   Section \ref{sec-four-star}.}
\end{example}

 The paper is organized as follows.
 In Section \ref{sec-basic}, we define optimal parametrization for general compact sets.
We introduce the  linear GIFS and the chain condition in Section \ref{sec-GIFS} and Section \ref{sec-chain}, respectively.
 Section \ref{sec-path} is devoted to the path-on-lattice IFS on the plane.
  Visualizations of  space-filling curves are discussed in Section \ref{sec-visual}.
 Section \ref{sec-four-star} studies the four-star tile.
In Section \ref{sec-proof}, we prove Theorem \ref{M-theo1}, using a measure-recording GIFS.

\section{\textbf{Optimal parameterizations of self-similar sets}}\label{sec-basic}
 Let $K\subset \R^d$ be a non-empty compact set. We call $K$ a self-similar set, if it is a union of small copies of itself, precisely, there exist similitudes $S_1,\dots, S_N: \R^d\to \R^d$ such that
$$
K=\bigcup_{j=1}^N S_j(K).
$$
In fractal geometry, the family $\{S_1,\dots, S_N\}$ is called an \emph{iterated function system}, or IFS in short;  $K$ is called the \emph{invariant set} of the IFS \cite{Hut81, Fal90}.
We denote by ${\mathcal H}^s$ the $s$-dimensional Hausdorff measure.
A set $E\subset \R^d$ is called  an \emph{$s$-set}, if
$0<{\mathcal H}^s(E)<\infty$  for some $s\geq 0$.

The IFS $\{S_1,\dots, S_N\}$ is said to satisfy the \emph{open set condition (OSC)}, if there is an open set $U$ such that
$\bigcup_{i=1}^N S_i(U)\subset U$ and the sets $S_i(U)$ are disjoint.
It is well-known that, if a self-similar set $K$ satisfies the {open set condition}, then it is an $s$-set. (See \cite{Fal90}.)

\begin{remark}\label{Schief}{\rm
If an IFS satisfies the OSC condition, and $\dim_H K$ equals the space dimension, then $K$ has non-empty interior (\cite{Shief94}),
and it is a self-similar tile. Especially, if the contraction ratios of $S_i$ are all equal to $r$, then $K$ is called a
\emph{reptile}. (In this case, we must have $r=1/\sqrt[d]{N}$, where $d$ is the dimension of the space.)
}
\end{remark}

Motivated by the  studies of the space-filling curves,
it is natural to define an optimal parametrization of more general sets (see \cite{Dai}).
Denote ${\mathcal L}$ the one-dimensional Lebesgue measure.

\begin{defi} {\rm Let $K\subset {\mathbb R}^d$ be an $s$-set.
 An  onto mapping   $\psi:[0,1]\rightarrow K$ is called an
\emph{optimal parametrization} of $K$ if the following three conditions are fulfilled.
\begin{enumerate}
  \item[($i$)]  $\psi$ is almost one-to-one, precisely, there exist $K'\subset K$ and $I'\subset [0,1]$ such that
  ${\mathcal H}^s(K\setminus K')={\mathcal L}([0,1]\setminus I')=0$ and $\psi:~I'\to K'$ is a bijection;

  \item[($ii$)] $\psi$ is measure-preserving in the sense that
  $$
  {\mathcal H}^s(\psi(F))=c{\mathcal L}(F) \text{ and } {\mathcal L}(\psi^{-1}(B))=c^{-1}{\mathcal H}^s(B),
  $$
  for any Borel set $F\subset [0,1]$ and any Borel set $B\subset K$, where $c={\mathcal H}^s(K)$.
  \item[($iii$)] $\psi$ is $1/s$-H\"older continuous, that is, there is a constant $c'>0$ such that
  $$
  |\psi(x)-\psi(y)|\leq c'|x-y|^{\frac{1}{s}} \  \text{ for all } x,y\in [0,1].
  $$

\end{enumerate}
}
\end{defi}

Our main concern  is: Does every connected self-similar set admit an optimal parametrization?
 According to the theorem of Mazurkiewicz-Hahn (\cite{Hans94}),
a set is the image of $[0,1]$ under a continuous mapping if and only if it is
compact, connected, and locally connected. We note that a connected self-similar set fulfills these conditions, since a self-similar set is locally connected as soon as it is connected (\cite{Hata85}).

\begin{figure}[h]
   \includegraphics[width=0.35\textwidth]{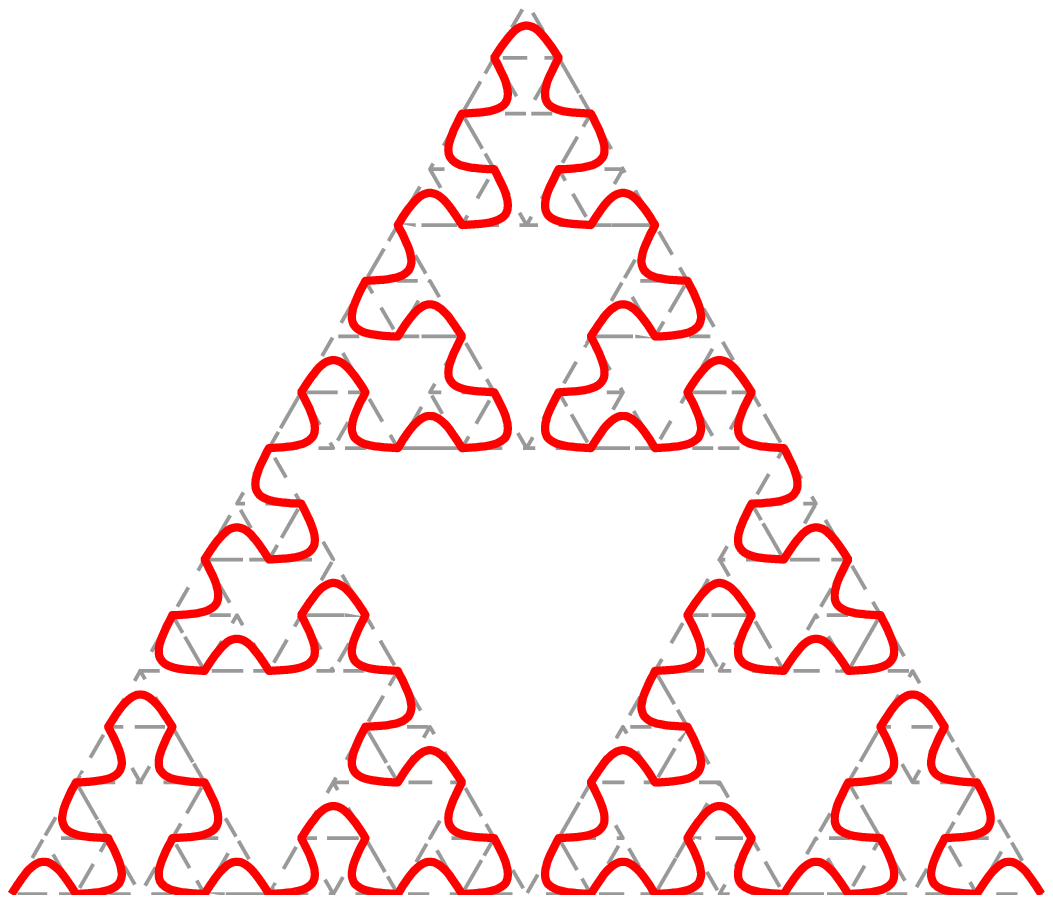}
   \includegraphics[width=.35 \textwidth]{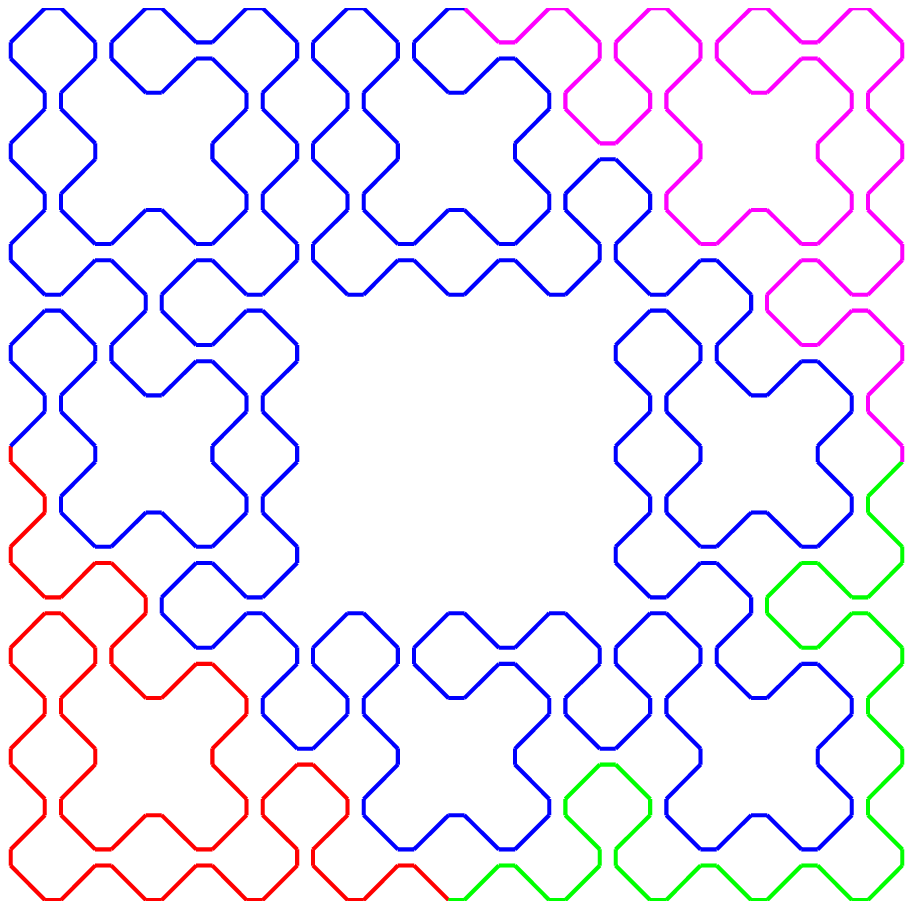}\\
 \caption{Optimal parameterizations of the Sierpi\'nski gasket and carpet. A detailed study of the carpet is carried out in \cite{Dai15}.}
 \label{fractal}
\end{figure}

For parameterizations of fractal sets, the previous studies  focused on the H\"older continuity.
de Rham \cite{deRham57}, Hata \cite{Hata85} and Remes \cite{Remes98} showed the existence of
$1/s$-H\"older continuous parameterizations for certain classes of self-similar sets.
Akiyama and Loridant \cite{Akiyama10,Akiyama11} showed the existence when $K$ is the boundary of a class of self-affine
tiles (their motivation is to provide an alternative way to show the disk-like property of some planar
tiles).
 Mart\'in and  Mattila \cite{MM2000} gave some negative results when $K$ is  disconnected.

\section{\textbf{Linear GIFS}}\label{sec-GIFS}

In this section,  we introduce  the notion of linear GIFS.

Let us start with the definition of GIFS. Let $G=(\mathcal{A},\Gamma)$ be a directed graph with vertex set $\mathcal{A}$ and edge set $\Gamma$. Let
 $${\mathcal G}=\left (g_{\boldsymbol \gamma}:\mathbb{R}^d\rightarrow \mathbb{R}^d \right )_{{\boldsymbol \gamma}\in \Gamma}$$
 be a family of similitudes. We call the triple $(\mathcal{A},\Gamma,{\mathcal G})$, or simply ${\mathcal G}$,  a \emph{graph-directed iterated function system} (GIFS).
 We  call $({\mathcal A}, \Gamma)$ the \emph{base graph} of the GIFS.
 Very often but not always,  we set ${\mathcal A}$ to be $\{1,\dots, N\}$.

Let $\Gamma_{ij}$ be the set of edges from state $i$ to $j$. It is well known that there exist unique non-empty compact sets $\{E_i\}_{i=1}^N$ satisfying
\begin{equation}\label{uniset}
E_i=\bigcup_{j=1}^N\bigcup_{{\boldsymbol \gamma}\in \Gamma_{ij}}g_{{\boldsymbol \gamma}}(E_j),\quad 1\leq i\leq N.
\end{equation}
We call $\{E_j\}_{j=1}^N$ the \emph{invariant sets} of the GIFS (\cite{MW88}\cite{Bedford86}).

We say the above GIFS satisfies the \emph{open set condition} (OSC), if there exist open sets $U_1,\dots,U_N$ such that
$$
\bigcup_{j=1}^N\bigcup_{\gamma\in\Gamma_{ij}}g_{\gamma}(U_j)\subset U_i,\quad 1\leq i\leq N,
$$
and the left-hand sides are non-overlapping unions (\cite{MW88}\cite{Fal97}).

\begin{remark}{\rm
 It is seen that the set equations \eqref{uniset} give all the information of a GIFS,
 and hence provide an alternative  way to define a GIFS.  We shall call \eqref{uniset} the \emph{set equation form} of a GIFS.
 }
\end{remark}


\subsection{Symbolic space related to a graph $G$} Let $G$ be a directed-graph.
 A sequence of edges in $G$, denoted by $\bomega=\omega_1\omega_2\dots\omega_n$, is called a \emph{path}, if the
  terminate state of $\omega_i$ coincides with the initial state of  $\omega_{i+1}$ for $1\leq i \leq n-1$.
We will use the following notations to specify the sets of finite or infinite paths on $G=(\mathcal{A},\Gamma)$. For $i\in \mathcal{A}$, let
 $$
 \Gamma_i^k,\ \Gamma_i^\ast \text{ and } \Gamma_{i}^{\infty}
 $$
  be the set of
 all paths with length $k$, the set of all paths with finite length, and the set of all infinite paths,
  emanating  from the state $i$, respectively.
Note that $\Gamma_i^\ast=\bigcup_{k\geq 1}\Gamma_i^k$.

 For a sequence $\bomega=(\omega_k)_{k=1}^\infty$, set $\bomega|_n=\omega_1\omega_2\dots\omega_n$ be the prefix of $\omega$ of length $n$.
For an infinite path ${\boldsymbol \omega}=(\omega_n)_{n=1}^\infty\in \Gamma_i^\infty$,
we  call
$$[\omega_1\dots\omega_n]:=\{{\boldsymbol \gamma}\in \Gamma_i^\infty;~~{\boldsymbol \gamma}|_n=\omega_1\dots\omega_n\}$$
the \emph{cylinder} associated with $\omega_1\dots \omega_n$.

For a path ${\boldsymbol \gamma}=\gamma_1\dots \gamma_n$, we denote
$$
E_{\boldsymbol \gamma}:=g_{\gamma_1}\circ\cdots \circ g_{\gamma_n}(E_{t(\bgamma)}),
$$
where $t(\bgamma)$ denotes the terminate state of the path $\boldsymbol \gamma$ (also $\gamma_n$). Iterating \eqref{uniset} $k$-times, we obtain
\begin{equation}\label{uni-set3}
E_i=\bigcup_{{\boldsymbol \gamma}\in \Gamma_i^k}E_{\boldsymbol \gamma}.
\end{equation}

We define a projection $\pi: (\Gamma_1^{\infty},\dots,\Gamma_N^{\infty}) \rightarrow (\mathbb{R}^d,\dots,\mathbb{R}^d)$, where $\pi_i: \Gamma_i^{\infty} \rightarrow \mathbb{R}^d$ is defined by
\begin{equation}\label{eq-projection}
\{\pi_i({\boldsymbol \omega})\}:=\bigcap_{n\geq 1}E_{\boldsymbol {\omega|_{n}}}.
\end{equation}
 For $x\in E_i$, we call ${\boldsymbol \omega}$ a coding of $x$ if $\pi_i({\boldsymbol \omega})=x$. It is folklore that $\pi_i(\Gamma_{i}^{\infty})=E_i$.

\subsection{\textbf{Order GIFS and linear GIFS}}
Let $({\mathcal{A}},\Gamma,\mathcal{G})$ be a GIFS. To study the `advanced' connectivity property of the invariant sets,
 we equip a partial order on the edge set $\Gamma$ enlightened by set equation \eqref{uni-set3}.
Let $\Gamma_i=\Gamma_i^1$ be the set of edges emanating from the vertex $i$.
\begin{defi}
{\rm

We call the quadruple $({\mathcal{A}},\Gamma, \mathcal{G}, \prec)$ an \emph{ordered GIFS},
if  $\prec$ is a partial order on $\Gamma$ such that
\begin{enumerate}
  \item[($i$)] $\prec$ is a linear order when restricted on $\Gamma_j$ for every $j\in {\mathcal A}$;
  \item[($ii$)]  elements in $\Gamma_i \text{ and } \Gamma_j$ are not comparable if  $i \neq j$.
\end{enumerate}
}
\end{defi}
We denote the edges in $\Gamma_i$ by $\gamma_{i,1},\gamma_{i,2},\dots,\gamma_{i,\ell_i}$ in an ascending order. For simplicity, we use the following equations to describe an order GIFS (equation form of an ordered GIFS):
$$
E_i= g_{\gamma_{i,1}}(E_{t(\gamma_{i,1})})+ g_{\gamma_{i,2}}(E_{t(\gamma_{i,2})})+\dots+
g_{\gamma_{i,\ell_i}}(E_{t(\gamma_{i,\ell_i})}), \quad i=1,\dots, N,
$$
where we use `$+$' instead of `$\cup$' to emphasize the order.

 The  order $\prec$  induces a \emph{dictionary order} on each
  $\Gamma_i^k$,
   namely,
$
\gamma_1\gamma_2\dots \gamma_k \prec \omega_1\omega_2\dots \omega_k
$
if and only if $ \gamma_1\dots \gamma_{\ell-1} = \omega_1\dots \omega_{\ell-1}$ and $\gamma_\ell\prec \omega_\ell$ for some $1\leq \ell \leq k$.
  Observe that $(\Gamma_i^k, \prec)$ is a linear order. Now we can define the linear GIFS.

\begin{defi}{\rm Let $({\mathcal{A}},\Gamma,\mathcal{G},\prec)$ be an ordered GIFS with invariant sets $\{E_i\}_{i=1}^N$.
 It is termed a \emph{linear} GIFS,  if for all $i\in {\mathcal A}$ and $k\geq 1$,
 $$
 E_{\boldsymbol \gamma}\cap E_{\boldsymbol \omega}\neq \emptyset
 $$
  provided ${\boldsymbol \gamma}$ and ${\boldsymbol \omega}$ are adjacent paths in
  $ \Gamma_{i}^k$.}
\end{defi}



\subsection{Linear IFS}
An IFS $\{S_1,\dots, S_N\}$ is a special class of GIFS, where the vertex set is a singleton, and the edge set consists of $N$ self edges which we denote by $1,\dots, N$. The IFS becomes an \emph{ordered IFS}, if we assume the natural order $1\prec 2\prec \cdots \prec N$.

 \begin{remark}{\rm For an IFS, the associated symbolic space is much simpler. Let $\Sigma=\{1,\dots,N\}$. For $m\geq1$, we denote $\Sigma^m=\{1,2\cdots N\}^m$,
$\Sigma^\ast=\bigcup_{m=0}^\infty{\Sigma^m}$, and $\Sigma^\infty=\{1,2\cdots N\}^{\mathbb{N}}$.
 For convention,  instead of calling
 $\omega_1\dots \omega_n\in \Sigma^n$ a path, we call it a word.
  For $i_1i_2\dots i_m\in \Sigma^*$, we  call
$[i_1i_2\dots i_m]=\{\bomega\in \Sigma^\infty ~;~~ \bomega|_m=i_1i_2\dots i_m\}$
a \emph{cylinder}.

Denote by $K$ the invariant set of the  IFS.
The \emph{projection map} $\pi:\Sigma^\infty\rightarrow K$ is
$$\{\pi(\bomega)\}=\bigcap_1^{\infty} S_{\bomega|_m}(K),$$
where $S_{i_1i_2 \dots i_m}=S_{i_1}\circ S_{i_2} \circ  \cdots \circ S_{i_m} $.
 We call $\bomega$ a \emph{coding} of $x$ if $\pi(\bomega)=x$.
}
\end{remark}

Clearly, the von Koch curve is generated by a linear IFS.
In Section \ref{sec-path}, we shall show that
the Peano curve is generated by a linear IFS,
while the Heighway dragon and the Hilbert curve are generated by linear GIFS'.

\begin{example}\label{ex-Sier} { Sierpi\'nski curve. \rm Let $T_1, T_2, T_3$ and $T_4$ be a partition of the unit square indicated by Figure \ref{SCurve}(right).
In Figure \ref{SCurve}(left),  $T_j$, $1\leq j\leq 4$, is divided into $4$ small triangles, where the numbers indicate the order of the small triangles. The four small triangles are images of $T_j$ under a map of the form $(x+b)/2$. Hence, we obtain
a linear GIFS. Precisely,
$$
T_1=\frac{T_1}{2}+\frac{T_2}{2}+\frac{T_4+1}{2}+\frac{T_1+1}{2}.
$$
Similarly equations can be obtained for $T_2,T_3$ and $T_4$.
\begin{figure}[h]
  \includegraphics[width=0.38\textwidth]{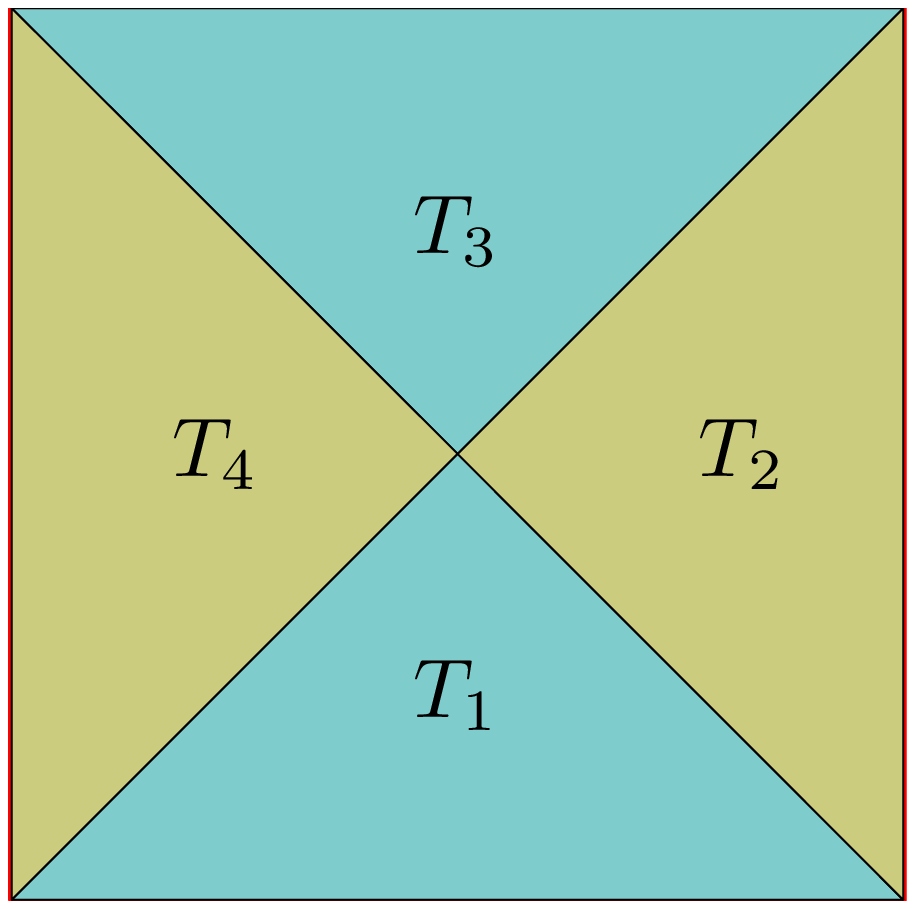}\quad
  \includegraphics[width=0.4\textwidth]{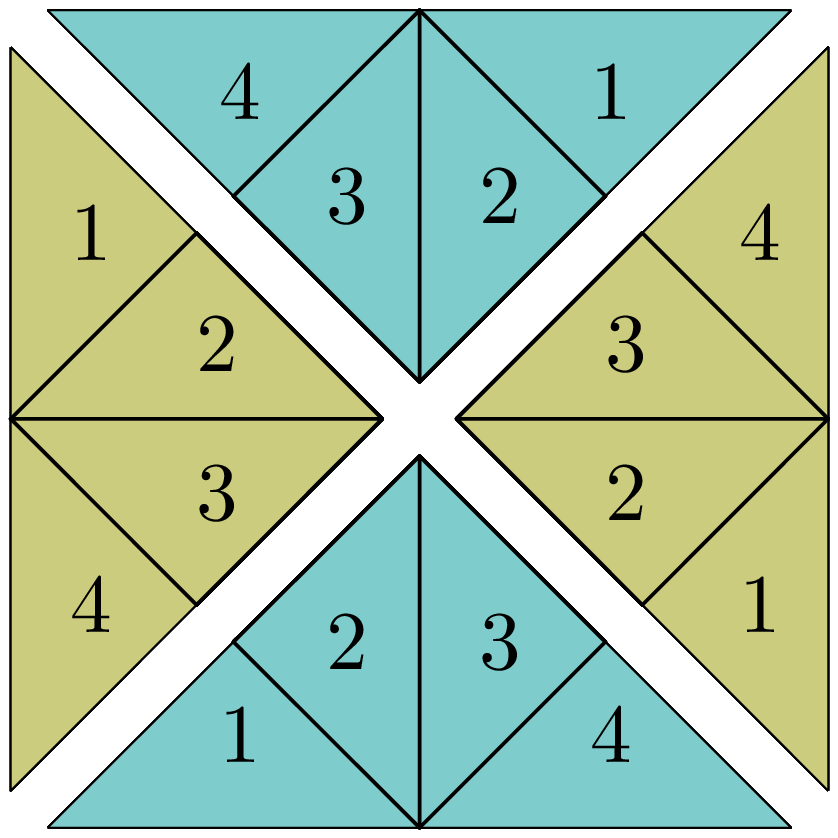}\\
  \caption{A linear GIFS which  generates the Sierpinski curve.}\label{SCurve}
\end{figure}

}
\end{example}

\section{\textbf{The chain condition, proof of Theorem \ref{theo-chain}}}\label{sec-chain}

Let $(\mathcal{A},\Gamma,{\cal G}, \prec)$ be an ordered GIFS. Denote the invariant sets by $\{E_i\}_{i\in {\mathcal A}}$.
For an edge $\omega \in \Gamma$, recall that $g_\omega$ is the associated similitude and $t(\omega)$ is the terminate state.

For $i\in \mathcal{A}$, a path $\bomega\in\Gamma_i^{\infty}$ is called the \emph{lowest} path, if $\bomega|_n$ is the lowest path in $\Gamma_i^n$ for all $n$; in this case, we call $a=\pi_i(\bomega)$ the \emph{head} of $E_i$.
Similarly, we define the highest path $\bomega'$ of $\Gamma_i^\infty$, and we call $b=\pi_i(\bomega')$ the
  the \emph{tail} of $E_i$.

\begin{defi}{\rm
An ordered GIFS is said to satisfy the \emph{chain condition}, if for any $i\in {\mathcal A}$, and any two adjacent edges
$\omega, \gamma\in \Gamma_i$ with $\omega \prec \gamma$,
$$g_\omega(\text{tail of } (E_{t(\omega)})=g_\gamma(\text{ head of }E_{t(\gamma)}).$$
}
\end{defi}

 Clearly, for an ordered IFS $\{S_1,\dots, S_N\}$,
 the lowest coding is $1^\infty$ and the highest coding is $N^\infty$. Therefore, the head of
 $K$ is the fixed point of $S_1$, denoted by $Fix(S_1)$,  and
 the tail of $K$ is $Fix(S_N)$. Consequently, the chain condition hold if and only if
\begin{equation}\label{chain-cond}
S_{i+1}( \text{Fix}(S_1))=S_{i}(\text{Fix}(S_N)) \text{ for }  i=1,2,\dots,N-1,
\end{equation}

Condition \eqref{chain-cond} first appeared in Hata \cite{Hata85}, when dealing with
 the  $1/s$-H\"older continuous  parametrization of self-similar sets.

\begin{proof}[\textbf{Proof of Theorem \ref{theo-chain}.}] Suppose $(\mathcal{A},\Gamma,{\cal G}, \prec)$ satisfies the chain condition. Let $i\in \mathcal{A}$,
and let $\bomega$ and $\bgamma$ be two adjacent paths in $\Gamma_i^n$ with $\bomega \prec \bgamma$. Let
$\bdeta=\bomega\wedge \bgamma$ be the largest common prefix of $\bomega$ and $\bgamma$. Then $\bomega$ and $\bgamma$ can be written as $\bomega=\bdeta \omega_{k+1}\dots \omega_n$ and $\bgamma=\bdeta \gamma_{k+1}\dots \gamma_n$.
The fact $\bomega$ and $\bgamma$ are adjacent implies that

(i) $\omega_{k+1}$ and $\gamma_{k+1}$ are adjacent edges  in $\Gamma_j$ where $j=t(\bdeta)$,
and $\omega_{k+1}\prec \gamma_{k+1}$;

 (ii) $\omega_{k+2}\dots \omega_n$ is the highest path in $\Gamma_{t(\omega_{k+1})}^{n-k-1}$ and $\gamma_{k+2}\dots \gamma_n$ is the lowest path in $\Gamma_{t(\gamma_{k+1})}^{n-k-1}$.\\
By item (ii),  we have
$b=\big(\text{tail of }E_{t(\omega_{k+1})}\big)\in E_{\omega_{k+2}\dots \omega_n}$ since the coding of $b$ is initialled
by $\omega_{k+2}\dots \omega_n$. Hence
$$g_{\omega_{k+1}}(\text{tail of }E_{t(\omega_{k+1})})\in  E_{\omega_{k+1}\dots \omega_n}.$$
Similarly,
$$g_{\gamma_{k+1}}(\text{head of }E_{t(\gamma_{k+1})})\in E_{\gamma_{k+1}\dots \gamma_n}.$$
Therefore $E_{\omega_{k+1}\dots \omega_n}\cap E_{\gamma_{k+1}\dots \gamma_n}\neq\emptyset$ by the chain condition. So
$$
E_{\bomega}\cap E_{\bgamma}=g_{\bdeta}(E_{\omega_{k+1}\dots \omega_n}\cap E_{\gamma_{k+1}\dots \gamma_n})\neq\emptyset,
$$
which proves that  the GIFS is linear.

On the other hand, assume that $({\mathcal{A},\Gamma,{\cal G},\prec})$ is a linear GIFS. Fix $i\in \mathcal{A}$. Let $\omega_1$ and $\gamma_1$ be  adjacent edges in $\Gamma_i$ satisfying $\omega_1\prec \gamma_1$.
Let $(\omega_k)_{k=2}^\infty$ be the highest path in $\Gamma_{t(\omega_1)}^{\infty}$ and $(\gamma_k)_{k=2}^\infty$ be the lowest path in $\Gamma_{t(\gamma_1)}^{\infty}$.
Denote $$\bomega|_{k}=\omega_1\dots\omega_k \text{ and } \bgamma|_{k}=\gamma_1\dots\gamma_k, $$ then for all $k\geq 1$, $\bomega|_{k}$ and $\bgamma|_{k}$ are adjacent path in $\Gamma_i^k$ and so $E_{\bomega|_{k}}\cap E_{\bgamma|_{k}}\neq \emptyset$.
As we know that
$$g_{\omega_1}(\text{tail of }E_{t(\omega_1)})=\pi_v((\omega_p)_{p\geq 1})\in E_{\bomega|_k},$$
$$g_{\gamma_1}(\text{head of }E_{t(\gamma_1)})=\pi_v((\gamma_p)_{p\geq 1})\in E_{\bgamma|_k}$$
 for all $k\geq 1$ , so
the distance between $g_{\omega_1}(\text{tail of }E_{t(\omega_1)})$ and $g_{\gamma_1}(\text{head of }E_{t(\gamma_1)})$ can be arbitrarily small. Thus
$$g_{\omega_1}(\text{tail of }E_{t(\omega_1)})=g_{\gamma_1}(\text{head of }E_{t(\gamma_1)}),$$
and the chain condition is verified. The theorem is proved.
\end{proof}

\begin{cor}\label{linearIFS} An ordered IFS $\{S_1,\dots, S_N\}$  is a linear IFS if and only if
 \eqref{chain-cond} holds.
\end{cor}



\section{\textbf{Path-on-lattice IFS on the plane}}\label{sec-path}

In this section, we study the path-on-lattice IFS on the plane (which we denote by ${\mathbb C}$).

 Let $\L=\Z+i\Z$ be the square lattice or $\L=\Z+\omega \Z$ be the triangle lattice in the plane,
 where $\omega=\exp(2\pi i/3)$.
 We define two points in $\L$ to be neighbors if their distance is $1$.
   Then we obtain a graph and we still denote it by $\L$.

 Let $P$ be a path in $\L$ passing through the points
 $0=z_0, z_1,\dots, z_{n-1}, z_n=d$ in turn.
 Let $\Phi=\{\phi_k\}_{k=1}^{n}$ be an ordered IFS on $\mathbb{C}$ such that
 \begin{equation}\label{path ifs}
 \phi_k(\{0,d\})=\{z_{k-1}, z_{k}\}, \text{ for all } k=1,\dots,n.
 \end{equation}
 We call such $\Phi$ a \emph{path-on-lattice IFS} with respect to the path $P$.
 Clearly the mapping $\phi_k$ has the form
 $\phi_k(z)=\alpha z+\beta, \text{ or } \phi_k(z)=\alpha \bar z+\beta$
 with $\alpha,\beta\in {\mathbb C}$, and there are four choices of $\phi_k$ for each $k$. If we indicate the four mappings by line segments with a half-arrow, then the IFS can be described by a path consisting of marked line segments.
 If all $\phi_k$ are of the form $\alpha z+\beta$, then we say $\Phi$ is
\emph{reflection-free}.

\begin{figure}[h]
   \includegraphics[width=.35 \textwidth]{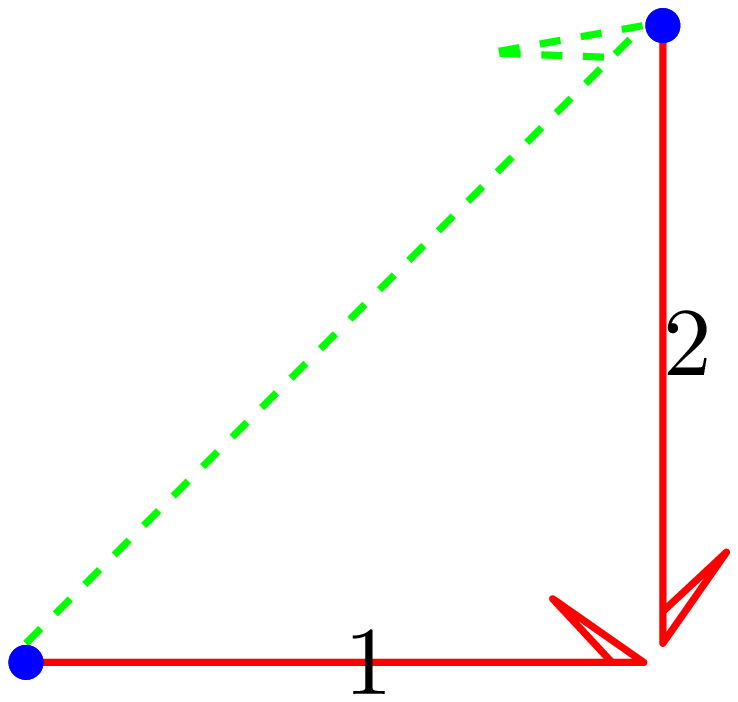}
  \includegraphics[width=0.33\textwidth]{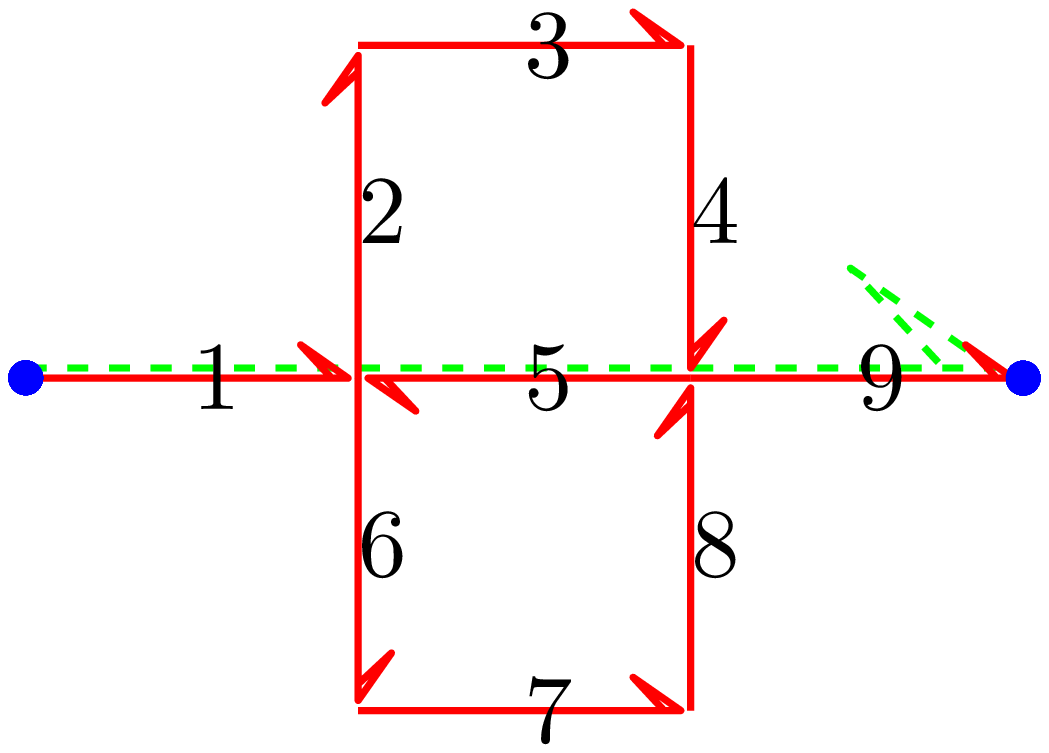}
  \includegraphics[width=0.28 \textwidth]{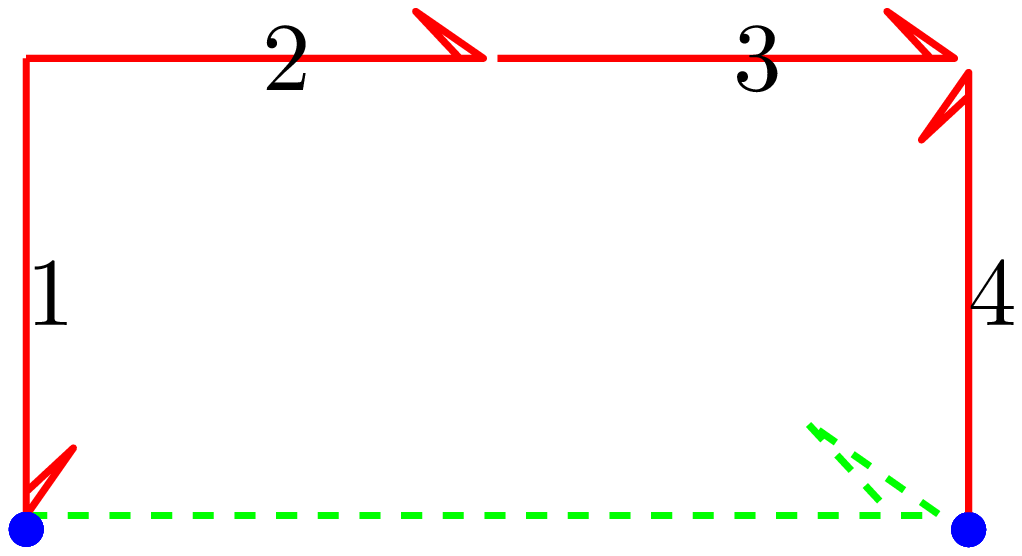}
  \caption{Paths for Heighway dragon curve, Peano curve and Hilbert curve.}\label{paths}
\end{figure}

 \begin{theorem}\label{theo-two-state} The path-on-lattice IFS is either a linear IFS,  or its invariant set
 can be generated by a linear GIFS with two states. Moreover, the linear GIFS satisfies the OSC if the original IFS
 does.
  \end{theorem}


\begin{proof}
 Let $\{\phi_k\}_{k=1}^n$ be a path-on-lattice IFS defined by \eqref{path ifs}.
Let $K$ be the invariant set.

 For $k=1,\dots, n$, define
 \begin{equation}\label{vv}
 v_{k}=\left \{
 \begin{array}{rl}
 1,    &\text{ if } (\phi_{k}(0),\phi_{k}(d))=(z_{k-1}, z_{k}),\\
  -1,  &\text{ if } (\phi_{k}(0),\phi_{k}(d))=(z_{k}, z_{k-1}).
  \end{array}
\right .
\end{equation}
We define an ordered GIFS with two state $\{1,-1\}$ as follows:
\begin{equation}\label{dragon1}
\left \{
\begin{array}{l}
E_1= \phi_1(E_{v_{1}})+\dots+ \phi_n(E_{v_{n}}),\\
E_{-1}= \phi_n(E_{-v_{n}})+\dots + \phi_1(E_{-v_{1}}).
\end{array}
\right .
\end{equation}
(One may think that the first equation is corresponding the path $P$, and the second equation is corresponding to
the reverse path of $P$.) Clearly $E_{1}=E_{-1}=K$.

We shall show that the head and tail of $E_1$ are $0$ and $d$ respectively, and the head and tail of $E_{-1}$
are $d$ and $0$ respectively; then using \eqref{vv}, we deduce that the GIFS \eqref{dragon1}
satisfies the chain condition.
According  to  $v_1=\pm 1$ and $v_n=\pm 1$,  we have four choices.

 Let us denote the $k$-th edge emanating from $E_i$ by $\lambda_{i,k}$.
 If $v_1=1$ and $v_n=1$, then the first edge emanating from the vertex $1$ is a self-edge, and hence $(\lambda_{1,1})^\infty$ is the lowest coding.
 It follows that the head of $E_1$ is $0$. Similarly, the highest coding emanating from vertex $1$
  is $(\lambda_{1,n})^\infty$, and so that the tail of $E_1$ is $d$.
 By the same argument,  the head of $E_{-1}$ is $d$ and
 the tail of $E_{-1}$ is $0$.

 The other three cases can be proved in the same manner. Moreover, if all $v_k$ equal  $1$, or all $v_k$ equal  $-1$, then
 GIFS \eqref{dragon1} degenerates to a linear IFS.

 As for the open set condition, if the original IFS satisfies the OSC with an open set $U$,
 then  GIFS \eqref{dragon1}  satisfies the OSC with open sets $\{U,U\}$.
 The theorem is proved.
\end{proof}

\begin{remark} {\rm Algorithms of checking the open set condition of path-on-lattice IFS are discussed in \cite{Yang16}.
For the path-on-lattice IFS \eqref{path ifs},
if $n=\|d\|^2$ and the OSC holds, then $K$ is a reptile (see Remark \ref{Schief}).
}
\end{remark}

\begin{example}\label{Ex-Peano} {\rm We give several space-filling curves generated by path-on-lattice IFS'.
All of them are reflection-free.
The visualizations will be explained in next section.

\emph{(1) Heighway dragon.}    The IFS is given by the path in Figure \ref{paths} (left), that is,
 $$\phi_1(z)=\frac{1-i}{2}z,~\phi_2(z)=-\frac{1+i}{2}z+(1+i).$$
Denote the Heighway dragon by $H$.
 By Theorem \ref{theo-two-state}, $v_1=1, v_2=-1$, and    $\{H,H\}$ are the invariant sets of the following  two-states linear GIFS:
 $$
\left \{
\begin{array}{c}
E_1= \phi_1(E_1)+ \phi_2(E_{-1})\\
E_{-1}= \phi_2(E_1)+ \phi_1(E_{-1}),
\end{array}
\right .
$$

 \emph{(2) Peano curve.} The IFS is given by the path in Figure \ref{paths}(middle).
 Then $v_1=\cdots =v_9=1$, and it is a linear IFS:
 $
 E=\phi_1(E)+\phi_2(E)+\cdots+\phi_9(E).
 $

\emph{(3) Hilbert curve.} The   IFS is given by the path in Figure \ref{paths}(right).
Clearly $v_1=v_4=-1, v_2=v_3=1$, and the corresponding linear GIFS is:
$$
\left \{
\begin{array}{c}
E_1= \phi_1(E_{-1})+ \phi_2(E_{1})+\phi_3(E_1)+ \phi_4(E_{-1})\\
E_{-1}= \phi_4(E_1)+ \phi_3(E_{-1})+ \phi_2(E_{-1})+ \phi_1(E_{1}).
\end{array}
\right .
$$
}
\end{example}

\begin{example}\label{Ex-Gosper}
{Gosper curve and anti-Gosper curve.
\rm The
 IFS of the Gosper curve is given by the path in Figure \ref{path-Gosper} (top-left).
 Clearly $v_1=v_4=v_5=v_6=1$, $v_2=v_3=v_7=-1$, 
 and the corresponding linear GIFS can be obtained accordingly.


If we forget the arrows of a path $P$, then we obtain a broken line and we call it the \emph{trace} of $P$.
It is shown in \cite{Yang16} that, among the   path-on-lattice IFS' which  are reflection-free and have the same trace as the Gosper curve (there are $128$ of them), none of them satisfies the open set condition except the Gosper curve and anti-Gosper curve. (The  path of the anti-Gosper curve is determined by
$v_1=v_2=v_5=1$ and $v_3=v_4=v_6=v_7=-1$. See Figure \ref{path-Gosper}(bottom).)

\begin{figure}[h]
  \centering
   \includegraphics[width=0.32\textwidth]{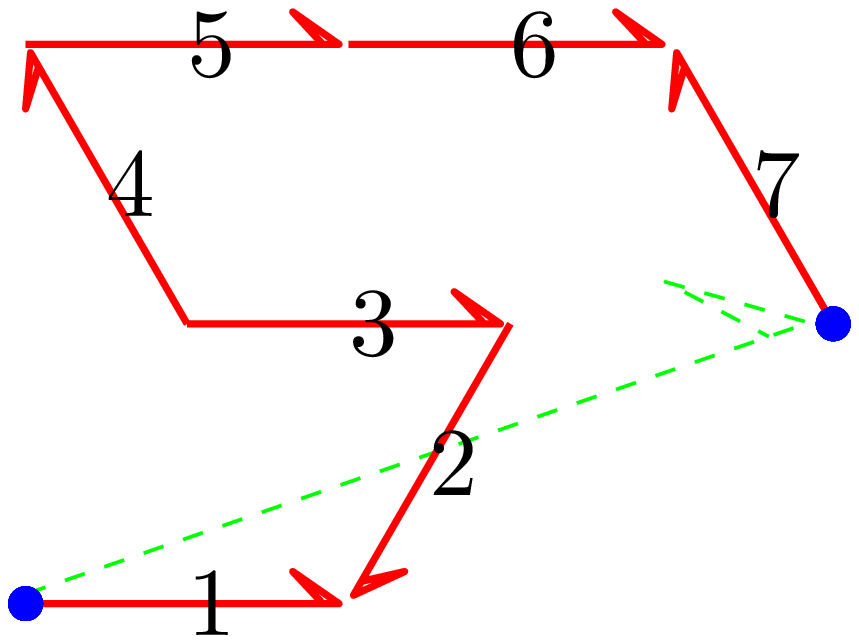}
    \includegraphics[width=0.31 \textwidth]{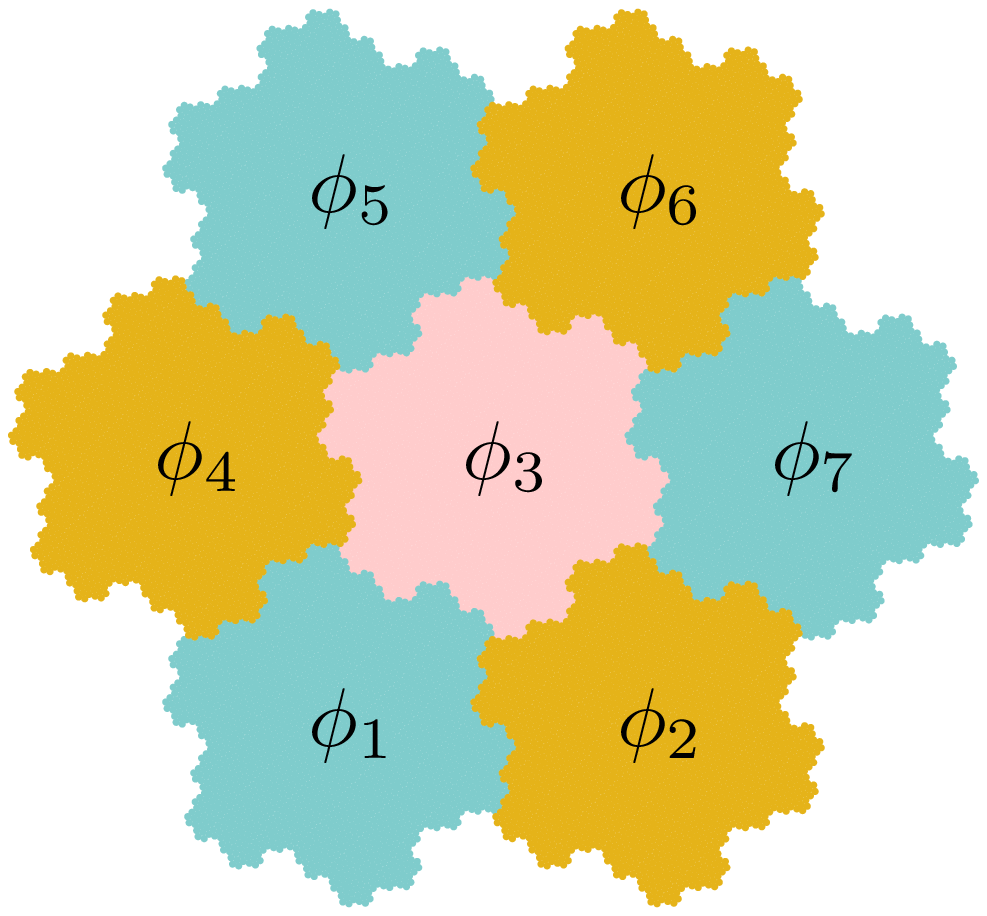}\\
    \ \ \includegraphics[width=0.25\textwidth]{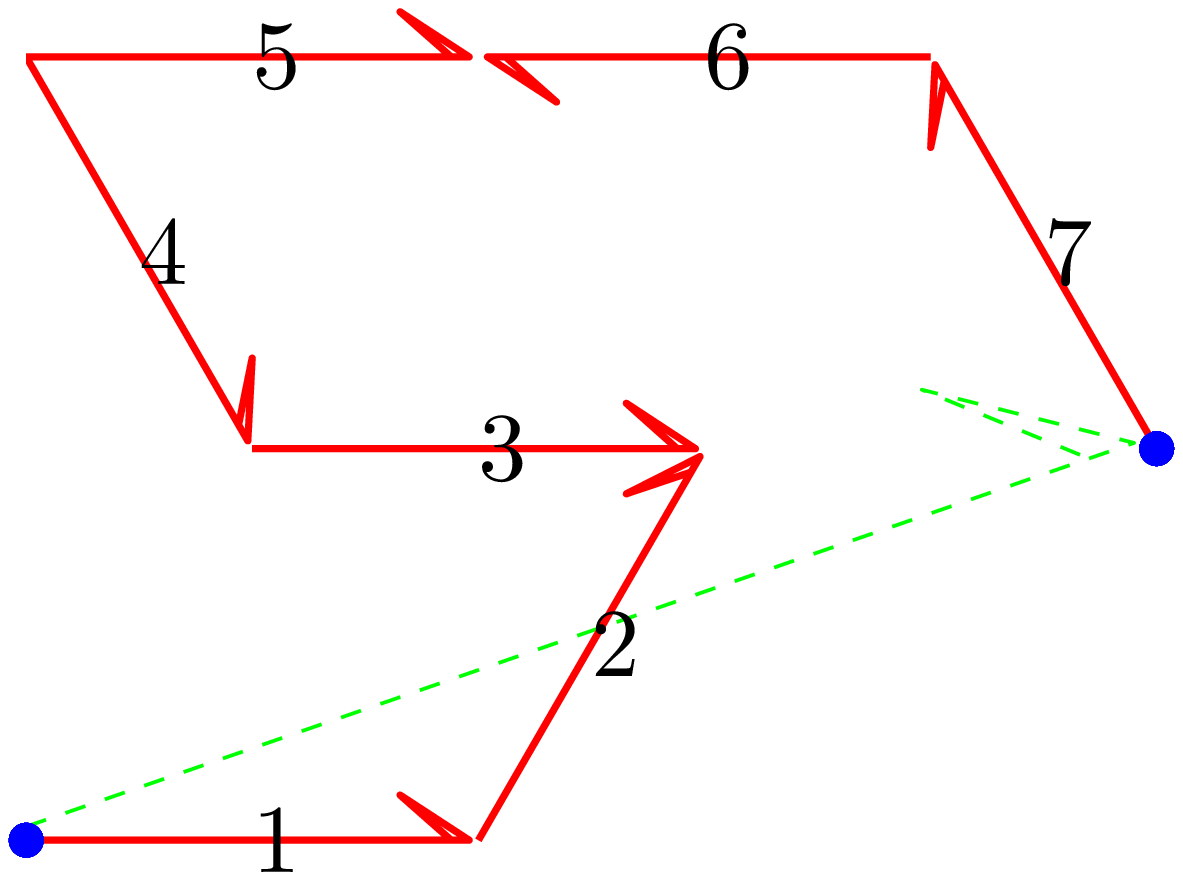} \quad \ \
    \includegraphics[width=0.31\textwidth]{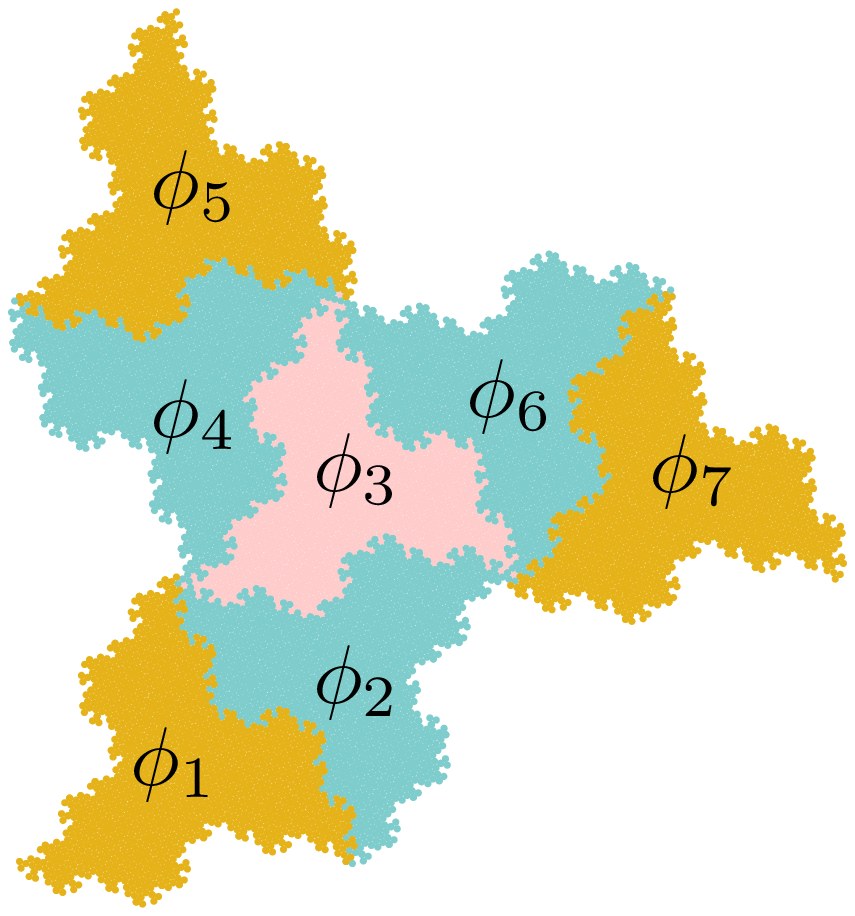}\\
  \caption{Marked paths for Gosper island and anti-Gosper island.}\label{path-Gosper}
\end{figure}
}
\end{example}




\section{\textbf{Visualizations  of space-filling curves}}\label{sec-visual}

We consider the visualizations of linear GIFS' in this section.

Let us start with linear IFS'.
Let $\mathcal{S}=\{S_i\}_{i=1}^N$ be a linear  IFS satisfying the open set condition.
According to Theorem 1.1,
an optimal parametrization $\varphi$ of $K$ can be  constructed accordingly.
To visualize  the limit curve $\varphi$, we need to choose an initial pattern; indeed, the initial pattern
 can be any curve, but a suitable choice will make the visualization beautiful.

Let us denote by  $L_0$ the initial pattern, and let $a$ and $b$ be  its initial and terminate point, respectively.
(Often $L_0$ is chosen to be a line segment, which we denote by $\overline{[a,b]}$.)

Fix an $n\geq 1$. For $\bomega\in \{1,\dots, N\}^n$, we denote by $x_\omega=S_\bomega(x)$, and denote
 $\bomega^+$ the follower of $\bomega$ (if $\bomega\neq (N)^n)$). Connecting $S_{\bomega}(a)$ and $S_{\bomega}(b)$ by $S_{\bomega}(L_0)$, and connecting
$S_{\bomega}(b)$ and $S_{\bomega^+}(a)$ by a line segment, we obtain the curve
$$L_n=\sum_{|{\boldsymbol \omega}|=n} \big(~S_{\boldsymbol \omega}(L_0)+ \overline{ [b_{\bomega}, a_{\bomega^+}]}~\big).$$
Here we use $\sum$ to indicate that $L_n$ is the joining of  small curves, where the order in given by $\bomega$.
We call  $L_n$ the $n$-th \emph{approximation} of the space-filling curve $\varphi$.
Different choices of initial patterns may give very different approximations in appearance, though the limit curve is the same.

\begin{example}\label{Approx-1}
{\rm
\emph{Peano curve.} The linear GIFS structure is  given in Example \ref{Ex-Peano}.
Figure \ref{PP01}(left and middle)  shows two different $2$nd approximations of the Peano curve. For the left picture,
the initial pattern $L_0=\{3/2\}$ is a singleton, and for the middle picture,  $L_0=\overline{[\epsilon, 3-\epsilon]}$ with $\epsilon=0.4$.

\begin{figure}[h]
  \includegraphics[width=0.33 \textwidth]{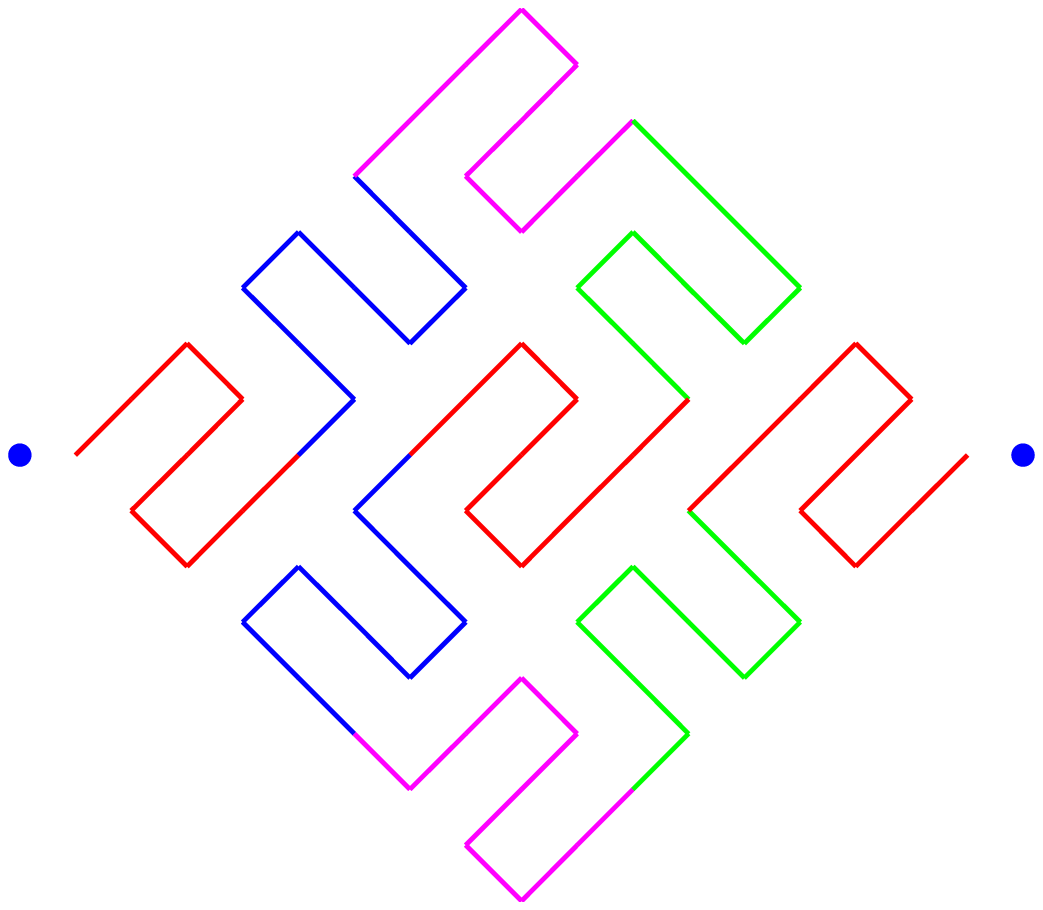}\includegraphics[width=0.33 \textwidth]{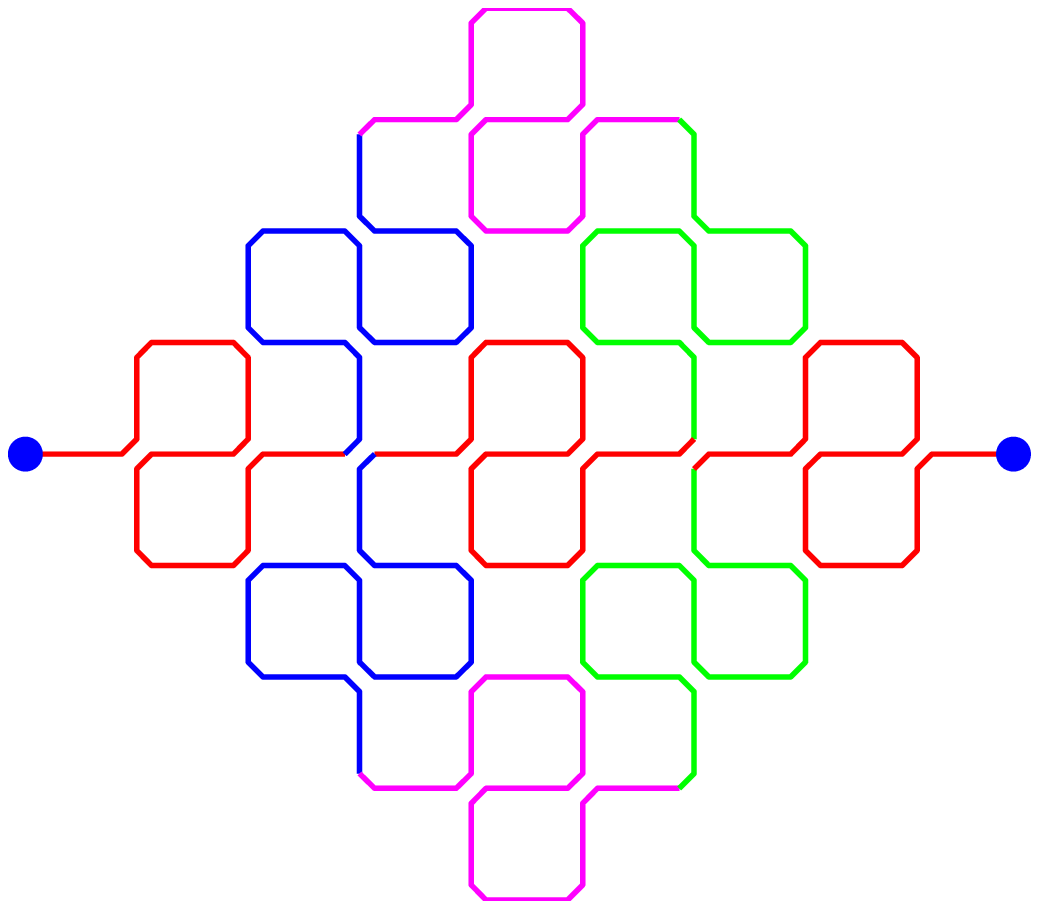}
   \includegraphics[width=0.33\textwidth]{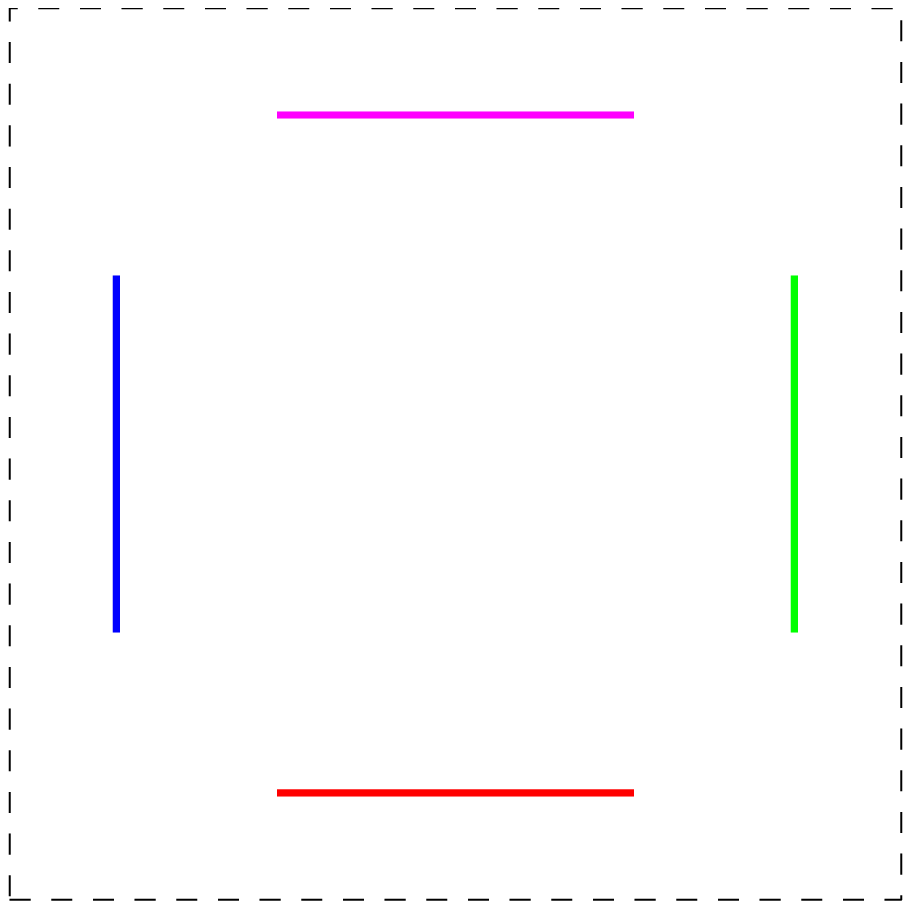}
  \caption{Left and middle: Visualizations of Peano curve. Right: Initial patterns
  of the Sierpi\'nski curve. The red, green, purple and blue line segments are initial patterns for $T_1,T_2,T_3$ and $T_4$, respectively.
 The dotted square is the unit square.}\label{PP01}
\end{figure}

 }
\end{example}

As for a linear GIFS, we have to choose an initial pattern for each $E_j$, which we denote by $L_0^1,\dots, L_0^N$.
Denote the initial point of the pattern $L_0^j$ by $a_j$ and the end point by $b_j$.
We define the $n$-th approximation of $E_j$ to be
$$L_n^j=\sum_{\bomega\in\Gamma_j^n} \big(~g_{\bomega}(L_0^{t(\bomega)})+\overline{[b_{\bomega},a_{\bomega^+}]}~\big),$$
where   $b_{\bomega}=g_{\bomega}(b_{t(\bomega)})$ and $a_{\bomega^{+}}=g_{\bomega^{+}}(a_{t(\bomega^{+})})$.

\begin{example} {\rm For  space-filling curves in this example, the linear GIFS structures are  given in Example  \ref{ex-Sier}, \ref{Ex-Peano} and \ref{Ex-Gosper}. The visualizations and initial patterns are listed in the following table\footnote
{Figure \ref{PHS}(middle) shows the Hilbert curve generated by the path-on-lattice IFS  multiplied by $1/2$.
 }.

\begin{tabular}[t]{l|l|l|l}
\hline
Hilbert curve & Fig. \ref{PHS}(middle) & 3rd &  $L_0^1=L_0^2=\{1+\mi\}$\\
Heighway dragon & Fig. \ref{fractal}(left) & 8th &  $L_0^1=L_0^2=\{(1+\mi)/2\}$\\
Gosper curve & Fig. \ref{pic-Gosper}(right) & 3rd & $L_0^1=\overline{[0, d]}, L_0^2=\overline{[d,0]}$ where $d=2+e^{\mi\pi/3}$\\
Sierpi\'nski curve & Fig. \ref{PHS}(right) & 3rd & Fig. \ref{PP01}(right)\\
\hline
\end{tabular}
 \smallskip
 }
 \end{example}

\section{\textbf{The four-star tile}}\label{sec-four-star}

The four-tile star is a $4$-reptile generated by the IFS
$$
S_1(z)=-\frac{z}{2}, \ S_2(z)=-\frac{z}{2}-\mi,
\ S_3(z)=-\frac{z}{2}+\exp(\mi\frac{5\pi}{6}), \ S_4(z)=-\frac{z}{2}+\exp(\mi\frac{\pi}{6}).
$$
Actually, the $S_j$ map the big dotted triangle to the four small triangles in Figure \ref{4star}.
  \cite{Gary} gave a visualization of the four-star tile without any explain (see Figure \ref{4star}). The  mathematical theory behind
are provided in \cite{Dai15}, and  here we give a sketch of it.

\subsection{\textbf{Skeleton}}
We introduce the notion of \emph{skeleton} of a self-similar set in \cite{Dai15}. (A skeleton is a kind of vertex set of a fractal.) It is shown that $\{a_1,\dots, a_6\}$ is a skeleton of the four-star tile, where
$
a_1=\frac{4}{3}\exp(\mi \frac{5\pi}{6}), \ a_2=\frac{2}{3}\mi, \ a_3=\frac{a_1}{\omega}, \ a_4=\frac{a_2}{\omega}, \
a_5=\frac{a_1}{\omega^2}, \ a_6=\frac{a_2}{\omega^2},
$
and $\omega=\exp(\mi 2\pi/3)$. (See Figure \ref{another} and \ref{4-star}.) 

\begin{figure}[h]
    \includegraphics[width=0.45\textwidth]{four_star_ske_03}
    \includegraphics[width=0.4\textwidth]{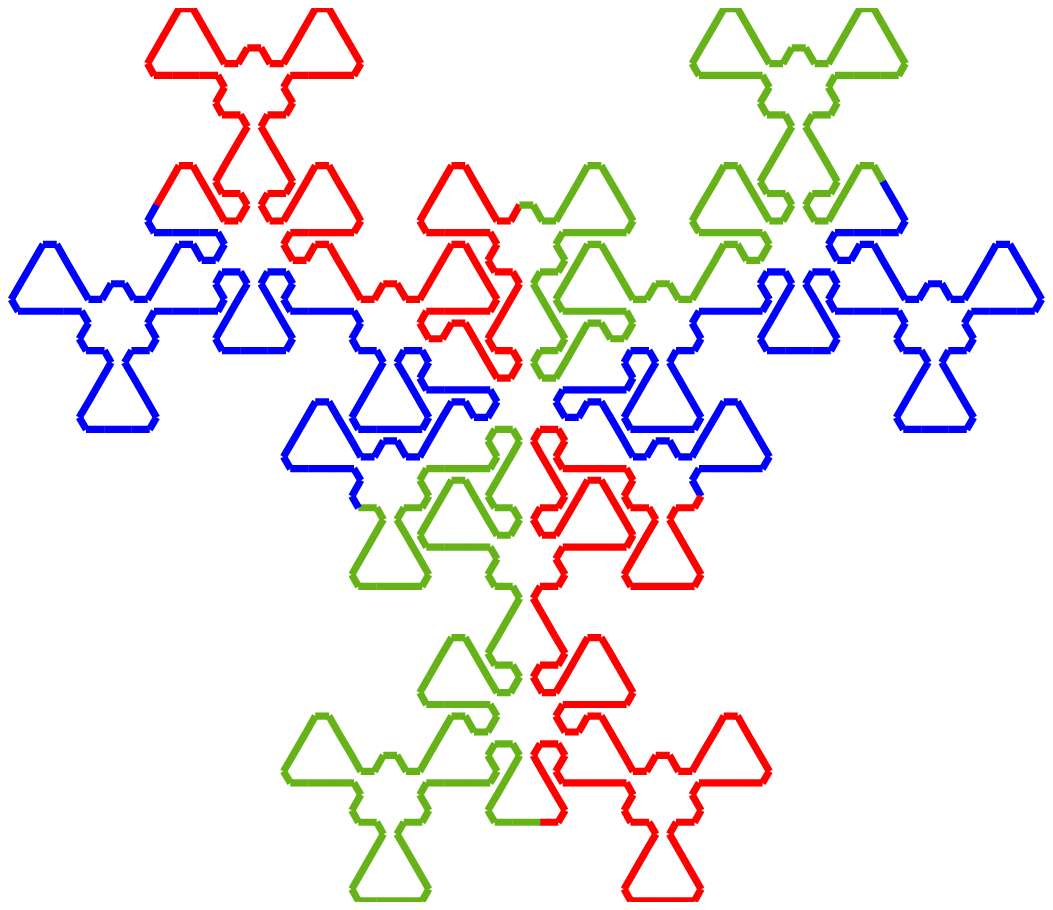}\\
  \caption{Another visualization of the four-star tile.}
  \label{another}
\end{figure}

\medskip
\subsection{\textbf{Substitution rule.}} Let us consider the polygon $P$ with vertices $a_1,a_2,\dots, a_6$.
Here we regard $P$ as a graph with directed edges, where each edge has the clock-wise orientation.
 Let us denote the edges by $x,y,z,u,v,w$. (See Figure \ref{4-star}(left).)
Figure \ref{4-star}(right) indicates  the graph
$$S_1(P)\cup S_2(P)\cup S_3(P)\cup S_4(P)$$
which  consist of $24$ directed edges.  A Eulerian path of the graph is indicated by Figure \ref{4-star}(right), and the path
is divided into $6$ parts indicated by different colors.

Replacing a segment in Figure \ref{4-star}(left) by the broken lines in Figure \ref{4-star}(right) with the same color, we obtain a substitution rule.

\begin{figure}[h]
  \includegraphics[width=0.34\textwidth]{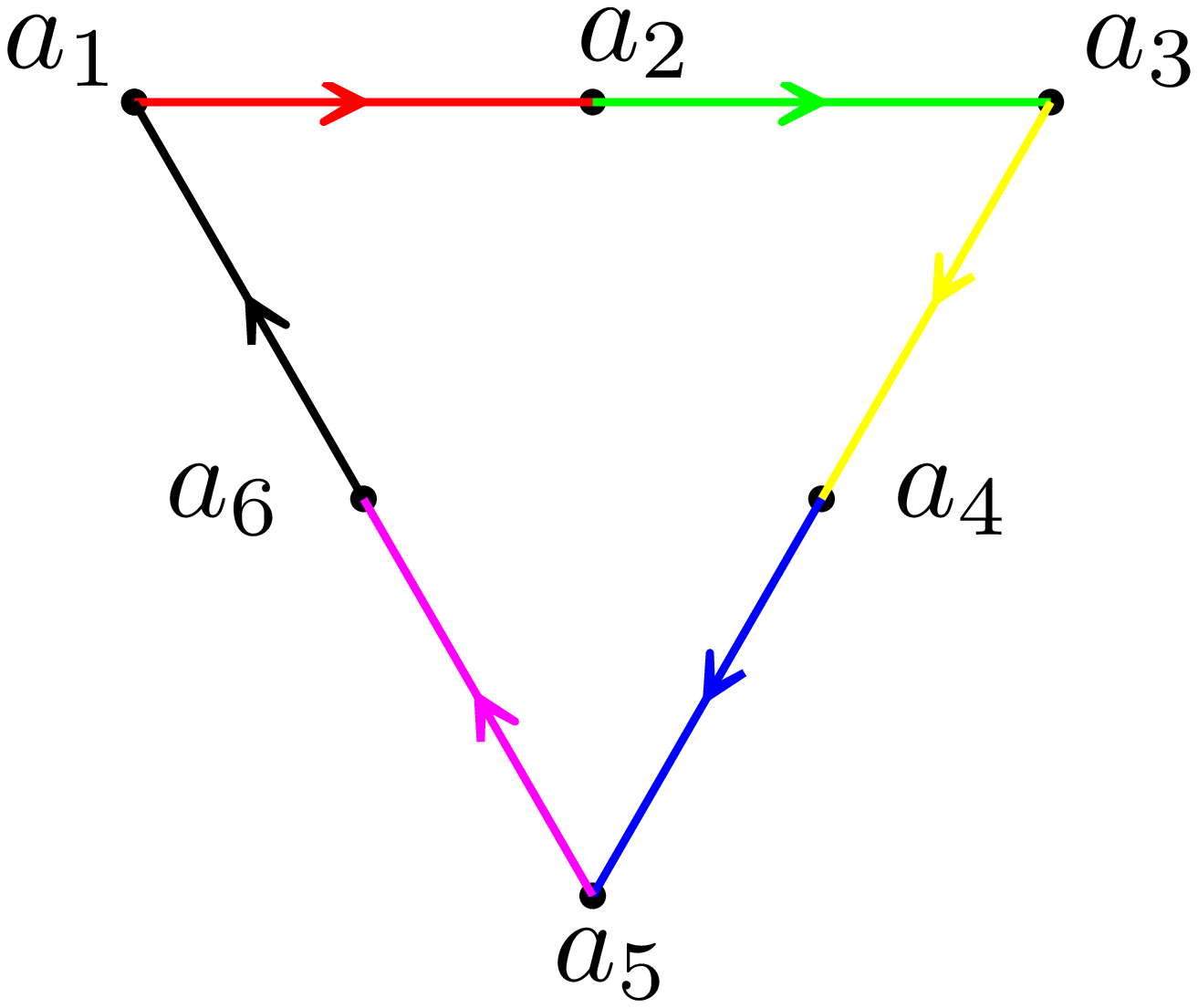}
  \includegraphics[width=0.39\textwidth]{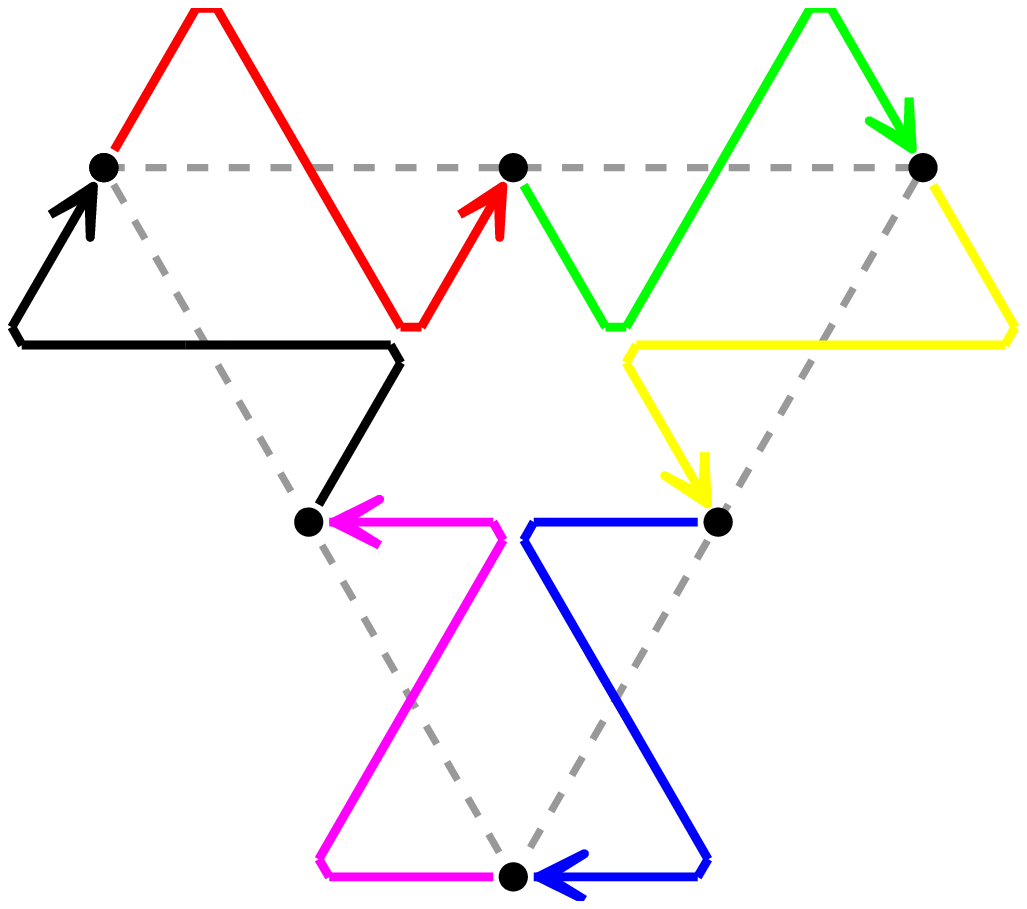}\\
  \caption{Substitution rule of the Four-star tile}\label{4-star}
\end{figure}

\medskip
\subsection{\textbf{Linear GIFS.}}
To precise the meaning of the substitution rule, we introduce the following  linear GIFS:
$$\left \{\begin{array}{l}
X=S_3(U)+S_3(V)+S_3(W)+S_1(U),\\
Y=S_1(V)+S_4(Z)+S_4(U)+S_4(V),\\
Z=S_4(W)+ S_4(X)+S_4(Y)+S_1(W),\\
U=S_1(X)+S_2(V)+S_2(W)+S_2(X),\\
V=S_2(Y)+S_2(Z)+S_2(U)+S_1(Y),\\
W=S_1(Z)+S_3(X)+S_3(Y)+S_3(Z).\\
\end{array}
\right .
$$
It is shown that the four-star tile coincides with $X\cup Y\cup Z\cup U \cup V\cup W$, and it is a non-overlapping union (in Lebesgue measure).

\subsection{\textbf{Visualizations.}}
 Figure \ref{ini pattern}(left) provides the initial patterns of the third visualization in
 Figure \ref{4star}. If we choose the initial patterns in  Figure \ref{ini pattern}(right), we obtain
 the visualization in Figure \ref{another}(right).

\begin{figure}[h]
  \includegraphics[width=.36\textwidth]{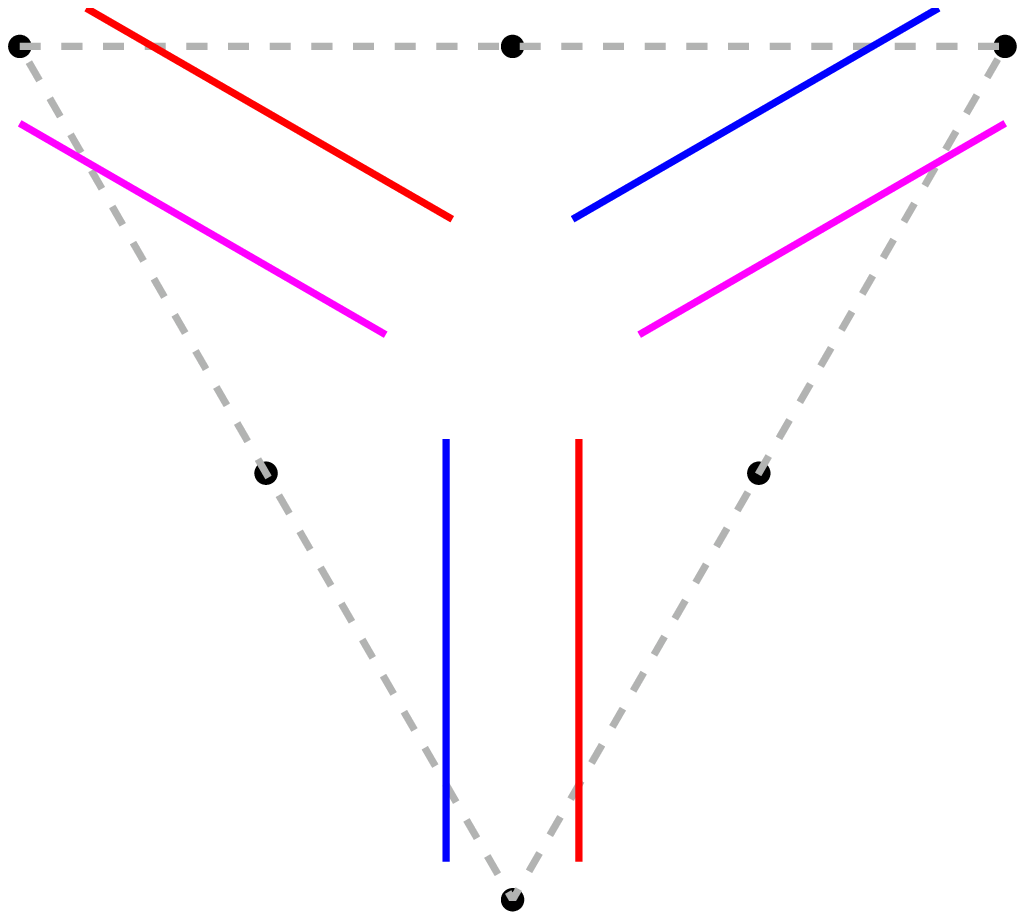}
   \includegraphics[width=.36\textwidth]{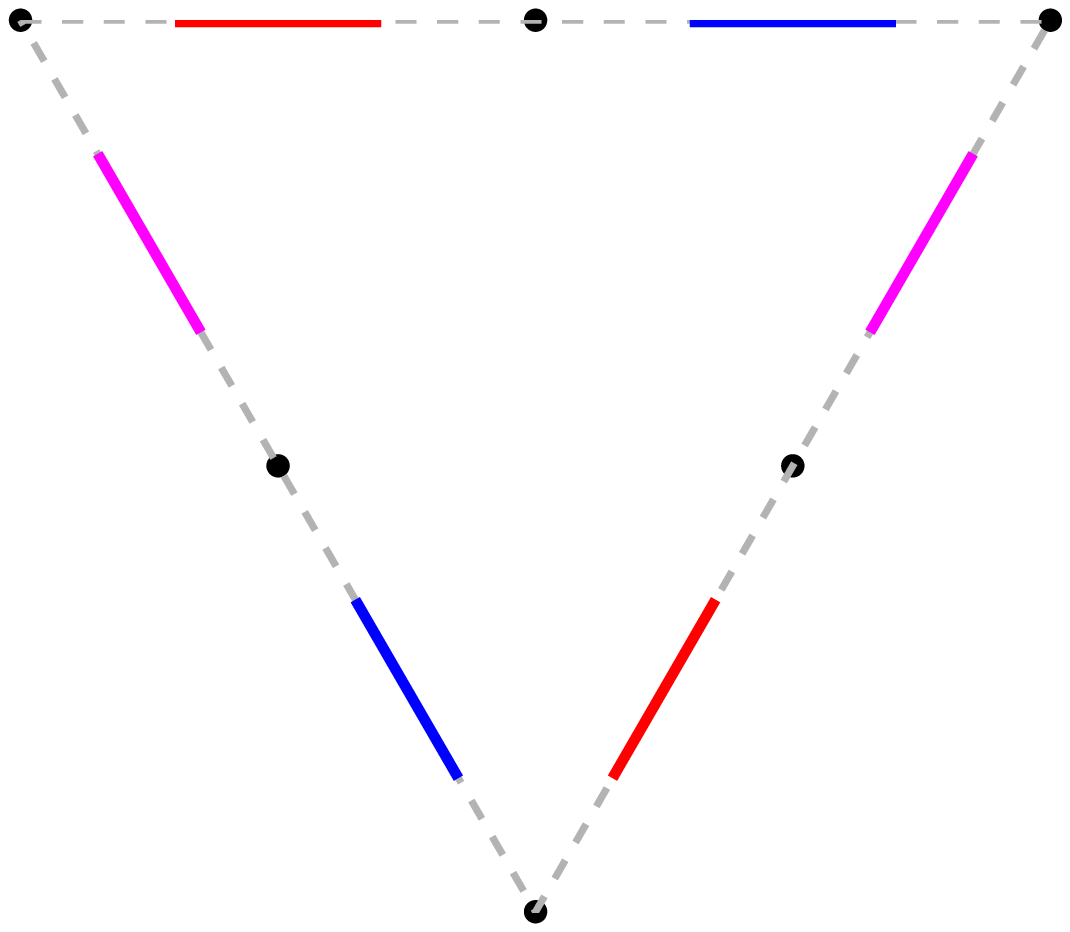}\\
  \caption{Initial patterns for four-star tile.}
  \label{ini pattern}
\end{figure}


\section{\textbf{Proof of Theorem \ref{M-theo1}; Measure-recording GIFS}}\label{sec-proof}
In this section, we show that the invariant sets of a linear GIFS with the open set condition admit optimal parameterizations.
 An auxiliary GIFS, called  the measure-recording GIFS, will play an important role.

\subsection{Preliminaries to dimensions and measures of graph-directed sets}
Let $(\mathcal{A},\Gamma,\mathcal{G})$ be the  GIFS given by \eqref{uniset}.

We say a directed-graph $({\mathcal A}, \Gamma)$ is \emph{strongly connected}, if there exists a path from $i$ to $j$ for any pair $i,j\in \mathcal{A}$.

Let $r_e$ denote the contraction ratio of the similitude $g_e$ associated with $e\in \Gamma$.
 Define a matrix $M(t)$, $t>0$,  as
\begin{equation}\label{matrix}
M(t)=\left (\sum_{e\in \Gamma_{ij}} r_e^t\right )_{1\leq i,j\leq N}.
\end{equation}
Then there exists a unique positive number $\delta$ , called the \emph{similarity dimension} of the GIFS, such that
$$\rho(M(\delta))=1$$
where $\rho(M)$ denotes the spectral radius of a matrix $M$.   (See \cite{MW88}, \cite{Fal97}.)

  \begin{theorem}\label{MW} (\cite{MW88})
  Let $(\mathcal{A},\Gamma,\mathcal{G})$ be a GIFS satisfying the OSC and let $\delta$ be the similarity dimension. Then

  (i) $\dim_H E_i=\delta$; and $0<H^\delta(E_i)<\infty$ for all $i=1,2,\dots,N$ if $\Gamma$ is strongly connected.

  (ii)  $(H^\delta(E_1),\dots, H^\delta(E_N))'$ 
  is an eigenvector of $M(\delta)$ corresponding to eigenvalue $1$. 

  (iii)  ${\mathcal H}^\delta(E_{\boldsymbol \omega}\cap E_{\boldsymbol \gamma})=0$ for any incomparable ${\boldsymbol \omega},{\boldsymbol \gamma} \in \Gamma_i^\ast$. (Two paths are said to be \emph{comparable} if one of them is a prefix of the other.)
\end{theorem}

%

In the rest of the section, we will always assume that ${\cal G}$ satisfies the OSC,
 and that  $0<\mathcal{H}^{\delta}(E_i)<\infty$ for all $i=1,\dots, N$.
 Let us denote
 $$
 h_i={\mathcal H}^\delta(E_i) \ \text{ and } \
 \mu_i={\mathcal H}^\delta|_{E_i},\quad i=1,\dots, N.
 $$

 Now, we define Markov measures on the symbolic spaces $\Gamma_i^\infty$, $i\in\mathcal{A}$.
 For an edge $e\in \Gamma$ such that $e\in \Gamma_{ij}$, set
\begin{equation}\label{our-weight}
p_e=\frac{h_j}{h_i}r_e^\delta.
\end{equation}
Using Theorem \ref{MW}(ii),
it is easy to verify that $(p_e)_ {e\in \Gamma}$ satisfies
\begin{equation}\label{weights}
\sum_{j\in {\mathcal A}}\sum_{e\in \Gamma_{ij}} p_e=1, \text{ for all } i\in \mathcal{A}.
\end{equation}
We  call $(p_e)_ {e\in \Gamma}$ a \emph{probability weight vector}.
 Let $\mathbb{P}_i$ be a Borel measure on $\Gamma_i^\infty$ satisfying the relations
\begin{equation}\label{measure}
\mathbb{P}_i([\omega_1\dots\omega_n])=h_i p_{\omega_1}\dots p_{\omega_n}
\end{equation}
for all cylinder $[\omega_1\dots\omega_n]$. The existence of such measures are guaranteed by \eqref{weights}. We call $\{{\mathbb P_i}\}_{i=1}^N$ the \emph{Markov measures} induced by the GIFS ${\mathcal G}$. The following result is folklore, see for instance \cite{MW88, L.Y2010}.

 \begin{theorem}\label{lem-Markov} Suppose the GIFS ${\mathcal G}$ satisfies the OSC and
  $0<h_i<+\infty$ for all $i$.
 Let $\pi_i: \Gamma_i^\infty\to E_i$ be the projections defined by \eqref{eq-projection}. Then
$$
\mu_i=\mathbb{P}_i\circ \pi_i^{-1}.
$$
 \end{theorem}



\subsection{\textbf{Measure-recording GIFS of a linear GIFS}}
Let $(\mathcal{A},\Gamma,\mathcal{G},\prec)$ be a linear GIFS such that the open set condition is fulfilled and
$0<h_i={\cal H}^\delta(E_i)<\infty$ for all $i$.
Set
$$
F_i=[0, h_i],\quad i=1,\dots, N.
$$
Fix a state $i$. We list the edges in $\Gamma_i$ in the ascendent order with respect to $\prec$:
$$
\gamma_1,\dots, \gamma_{\ell_i}.
$$
 Recall that $t(\gamma)$ denotes the terminate state of an edge $\gamma$. Then according to
the set equation form of  ${\cal G}$, $E_i$ can be written as
$$
E_i= g_{\gamma_1}(E_{t(\gamma_1)})+ \cdots + g_{\gamma_{\ell_i}}(E_{t(\gamma_{\ell_i})}).
$$
Let
$$f_{\gamma_k}(x)=r_{\gamma_k}^{\delta}x+b_k: ~~\mathbb{R}\longrightarrow\mathbb{R}, \quad 1\leq k \leq \ell_i$$
be similitudes such that
\begin{equation}\label{partition-F}
F_i= f_{\gamma_1}(F_{t(\gamma_1)})+ \cdots + f_{\gamma_{\ell_i}}(F_{t(\gamma_{\ell_i})})，
\end{equation}
where the right hand side is a non-overlapping union of consecutive intervals from left to right. Indeed,
we must have $b_k=\sum_{j=1}^{k-1}h_{t(\gamma_j)}r_{\gamma_j}^\delta$, and \eqref{partition-F} holds by equation \eqref{weights}.
Doing this for all $i\in {\mathcal A}$, then \eqref{partition-F} give us an ordered GIFS with the natural order. We denote this GIFS by
$$
(\mathcal{A}, \Gamma, {\cal G}^*, \prec),
$$
and call it the \emph{measure-recording GIFS} of $({\mathcal A}, \Gamma, {\mathcal G}, \prec)$.

Clearly, the measure-recording GIFS inherits the graph structure and the order structure of the original GIFS;
 moreover,  it records the Hausdorff measure information
of the original GIFS.
The following facts are obvious.
\begin{itemize}
  \item $\{F_i\}_{i=1}^N$ are the invariant sets of the measure-recording  GIFS.
  \item For an edge $e\in \Gamma$, the contraction ratio of $f_e$ is $r_e^\delta$, and the similarity dimension
   $\delta^*$ of ${\cal G}^*$ is $1$.
  \item ${\mathcal G}^*$ satisfies the OSC.
  \item The measure-recording GIFS shares the same symbolic spaces with the original GIFS.
\end{itemize}

Let
$$
\pi_i: \Gamma_i^\infty \rightarrow E_i \text{ and }
\rho_i: \Gamma_i^\infty \rightarrow F_i,\quad i=1,\dots, N,
$$
 be  projections w.r.t. the GIFS $({\cal G})$ and $({\cal G}^*)$, respectively. (See \eqref{eq-projection}.) Then

 \begin{lemma}\label{Markovmeasure}
  The Markov measure induced by the measure-recording GIFS coincides with that induced by the original GIFS.
 \end{lemma}

 \begin{proof} Let $(p_e)_{e\in \Gamma}$ and $(p_e^*)_{e\in \Gamma}$ be the probability weights corresponding to ${\cal G}$ and ${\cal G}^*$, respectively. Since $$p_e=\frac{h_j}{h_i}(r_e)^\delta=\frac{\mathcal{L}(F_j)}{\mathcal{L}(F_i)}(r_e^\delta)^{\delta^*}=p_e^*,$$
the two systems define the same probability weight vector and hence define the same Markov measure.
 \end{proof}

  Define
\begin{equation}\label{psi-map}
\psi_i:=\pi_i\circ \rho_i^{-1}.
\end{equation}
The following lemma verifies  that $\psi_i$ is a well-defined mapping from $F_i$ to $E_i$.


\begin{figure}[h]
  \centering
  \includegraphics[width=.4\textwidth]{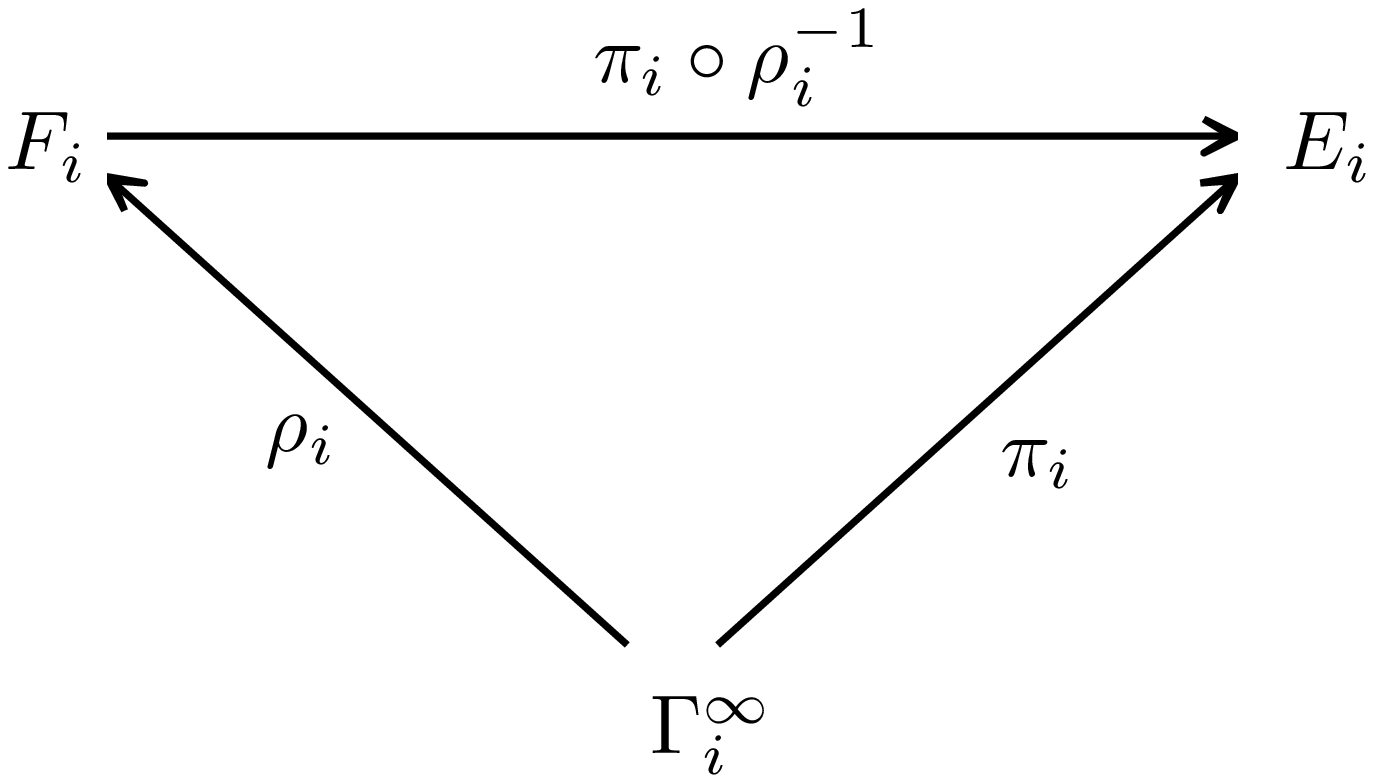}
 \vspace{-.9 cm}
\end{figure}

\begin{lemma}
Suppose $x\in F_i$ has  two $\rho_i$-codings, say $\rho_{i}^{-1}(x)=\{{\boldsymbol \omega}, \bgamma\}.$
Then
$$\pi_i(\bomega)=\pi_i(\bgamma).$$
\end{lemma}

\begin{proof}
 Write ${\boldsymbol \omega}=(\omega_k)_{k=1}^\infty$ and ${\boldsymbol \gamma}=(\gamma_k)_{k=1}^\infty$.
 We claim that  $\omega_1\dots \omega_{n}$ and $\gamma_1\dots \gamma_n$ are adjacent for all $n\geq 1$,
 for otherwise, there exists $\eta_1\dots \eta_n$ such that
  $$
  \omega_1\dots \omega_n\prec \eta_1\dots \eta_n\prec \gamma_1\dots \gamma_n,
  $$
  and the interval $F_{\eta_1\dots \eta_n}$ separates $F_{\omega_1\dots \omega_n}$ and $F_{\gamma_1\dots \gamma_n}$,
  contradicting to $\rho_i(\bomega)=\rho_i(\bgamma)=x$. Our claim is proved. It follows that
 $E_{ \omega_1\dots \omega_n }\cap E_{\gamma_1\dots \gamma_n}\neq \emptyset$, since $(\mathcal{A},\Gamma, \mathcal{G}, \prec )$ is a linear GIFS. Hence $|\pi_i(\bomega)-\pi_i(\bgamma)|\leq \text{diam } E_{\bomega|_n}+\text{diam } E_{\bgamma|_n}$,  so the distance between $\pi_i({\boldsymbol \omega})$ and $\pi_i({\boldsymbol \gamma})$ can be arbitrarily small,
   which implies that
$\pi_i({\boldsymbol \omega})=\pi_i({\boldsymbol \gamma})$.
\end{proof}

Now, we prove Theorem \ref{M-theo1} by showing that the mapping $\psi_i$ is an optimal parametrization of $E_i$.
\medskip

\subsection {Proof of Theorem \ref{M-theo1}.} Let ${\cal G}^*$ be the measure-recording GIFS of ${\cal G}$.
Let $\psi_i=\pi_i\circ\rho_i^{-1}.$ Let
 $\nu_i={\mathcal L}|_{F_i}$ be the restriction of the Lebesgue measure on $F_i$, $\mu_i={\mathcal H}^s|_{E_i}$,
 and  ${\mathbb P}_i$ be the common Markov measure
of ${\mathcal G}$ and ${\mathcal G}^*$.
Then  $\nu_i=\mathbb{P}_i\circ\rho_i^{-1}$, $\mu_i=\mathbb{P}_i\circ\pi_i^{-1}$ by Theorem \ref{lem-Markov}.

(i) First, we prove that $\psi_i$ is almost one to one.

Let $Q_i$ be the set of points in $E_i$ possessing more than one $\pi_i$-codings.
Since 
$$
\displaystyle
Q_i=\bigcup_{n \geq 1}\bigcup_{{\bgamma\neq \bomega \in \Gamma_i^n}} E_{\boldsymbol \gamma}\cap E_{\boldsymbol \omega},
$$
and $\mu_i(E_{\bgamma}\cap E_{\bomega})=0$  by Theorem  \ref{MW}(iii), we obtain  $\mu_i(Q_i)=0$.
Denote
$$
\Delta_i=\pi_i^{-1}(Q_i),
$$
then $\pi_i$ is injective when restricted to  $\Gamma_i^\infty\setminus \Delta_i$, and ${\mathbb P}_i(\Delta_i)=\mu_i(Q_i)=0$.

Similarly, let $Q_i'$ be the set of points in $F_i$ possessing more than one $\rho_i$-codings, then $\nu_i(Q_i')=0$.
Let $\Delta_i'=\rho_i^{-1}(Q_i')$, then $\rho_i$ is injective when restricted to  $\Gamma_i^\infty\setminus \Delta_i'$, and
${\mathbb P}_i(\Delta_i')=\nu_i(Q_i')=0$.

Let
$$
F_i'=\rho_i(\Gamma_i^\infty \setminus(\Delta_i\cup \Delta_i')), \quad
E_i'=\pi_i(\Gamma_i^\infty \setminus(\Delta_i\cup \Delta_i')),
$$
Then $\psi_i: F_i'\to E_i'$ is one-to-one, $\nu_i(F_i\setminus F_i')\leq\mathbb{P}_i(\Delta_i\cup\Delta_i')=0$, and $\mu_i(E_i\setminus
E_i')\leq\mathbb{P}_i(\Delta_i\cup\Delta_i')=0$.

(ii) Secondly, we prove that $\psi_i$ is measure-preserving.
 For any Borel set $B \subset \mathbb{R}^d$, we need to show that $\nu_i(\psi_i^{-1}(B))=\mu_i(B)$; due to (i), it suffices to show
 this hold for $B\subset E_i'$.  Indeed, for $B\subset E_i'$, we have
 $$
\nu_i(\psi_i^{-1}(B))=\nu_i(\rho_i\circ \pi_i^{-1}(B))=\mathbb{P}_i\circ\rho_i^{-1}\circ\rho_i\circ \pi_i^{-1}(B)=\mathbb{P}_i\circ \pi_i^{-1}(B)=\mu_i(B),
$$
where the third equality holds since $\rho_i$ is a bijection when restricted to
$(\Gamma_i^\infty \setminus(\Delta_i\cup \Delta_i'))$.
Similarly, for any Borel set $B\subset \R$, one can show that $\mu_i(\psi_i(B))=\nu_i(B).$

(iii) Finally, we prove the $1/\delta$-H\"older continuity of $\psi_i$.

Let $x_1,x_2$ be two points in $F_i=[0,h_i]$. Let $k$ be the smallest integer such that $x_1, x_2$ belong to two different cylinders of rank $k$, say,  $x_1\in \rho_i([{\boldsymbol \omega}])$,
$x_2\in \rho_i([{\boldsymbol \gamma}])$, where ${\boldsymbol \omega}\neq {\boldsymbol \gamma}\in \Gamma_i^k$.
It is seen that ${\boldsymbol \omega}=\omega_1\dots \omega_k$ and ${\boldsymbol \gamma}$ differ only at the last edge,
that is,
$$
\bgamma=\omega_1\dots\omega_{k-1}\gamma_k.
$$
We consider two cases according to ${\boldsymbol \omega}$ and ${\boldsymbol \gamma}$ are adjacent or not.

\emph{Case 1.} ~ ${\boldsymbol \omega}$ and ${\boldsymbol \gamma}$ are not adjacent. (See figure \ref{unadjacent}.)

\begin{figure}[h]
  \centering
  \includegraphics[width=1 \textwidth]{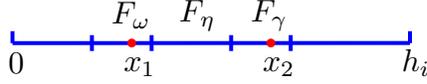}
 \vspace{-1.0 cm} \caption{ ${\bomega}$ and ${\bgamma}$ are not adjacent} \label{unadjacent}
\end{figure}

 Then there is a cylinder
$\bdeta=\omega_1\dots \omega_{k-1}\eta_k$ between $\bomega$ and $\bgamma$, so
$$
|x_1-x_2|\geq \text{diam}~F_{\boldsymbol \eta}\geq
h \cdot r_{{\boldsymbol \eta}}^\delta
\geq h  \cdot r_{\bomega^*}^\delta \cdot r_{\min}^\delta,
$$
where ${\boldsymbol \omega}^*=\omega_1\dots \omega_{k-1}$ is the path obtained by deleting the last edge in ${\boldsymbol \omega}$,
and
$$h=\min \{h_i;~i=1,\dots, N\}, \quad r_{\min}=\min\{r_e;~e\in \Gamma\}.$$
Since  $x_1,x_2$ belong to $\rho_i([{\boldsymbol \omega}^*])$,
 the images of $x_1$ and $x_2$ under $\pi_i\circ\rho_i^{-1}$, which we denote by $y_1$ and $y_2$ respectively,
belong to $\pi_i([{\boldsymbol \omega}^*])=E_{{\boldsymbol \omega}^*}$.
It follows that
\begin{equation}\label{Holder-1}
|y_1-y_2|\leq \text{diam } E_{{\boldsymbol \omega}^*}\leq D \cdot r_{{\boldsymbol \omega}^*}\leq
D \cdot r_{\min}^{-1}\cdot h^{-1/\delta} \cdot |x_1-x_2|^{1/\delta},
\end{equation}
where $$D=\max_{1\leq i\leq N} \text{ diam } E_i.$$

\medskip
\emph{Case 2.} ~ ${\boldsymbol \omega}$ and ${\boldsymbol \gamma}$ are adjacent. (See figure \ref{adjacent}(left).)

\begin{figure}[h]
  \centering
  \includegraphics[width=1 \textwidth]{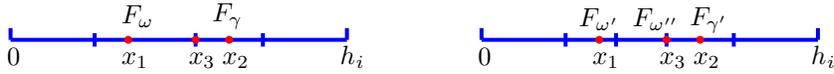}
  \vspace{-0.7 cm}\caption{${\bomega}$ and ${\bgamma}$ are adjacent}\label{adjacent}
\end{figure}

Let $x_3$ be the intersection of $F_\bomega$ and
 $F_\bgamma$. Let $k'$ be the smallest integer such that $x_1$ and $x_3$ belong to different cylinders of rank $k'$, say, $x_1\in \rho_i([\bomega'])$ and $x_3\in \rho_i([\bomega''])$ (see Figure \ref{adjacent}(right)), then
 $|x_1-x_3|\geq \text{diam}~F_{\bomega''}$
 since $x_3$ is an endpoint.

 Let $y_3=\psi_i(x_3)$. Similar to Case 1,  we have
 $$
 |y_1-y_3|\leq
D\cdot  r_{\min}^{-1} \cdot h^{-1/\delta} \cdot |x_1-x_3|^{1/\delta}.
$$
By the same argument, we have
 $$
 |y_2-y_3|\leq
D\cdot r_{\min}^{-1} \cdot h^{-1/\delta} \cdot |x_2-x_3|^{1/\delta}.
$$
Hence, by the fact $x_3$ locates between $x_1$ and $x_2$,
\begin{equation}\label{Holder-2}
 |y_1-y_2|\leq 2D\cdot r_{\min}^{-1} \cdot h^{-1/\delta} \cdot |x_1-x_2|^{1/\delta}.
\end{equation}
Therefore, \eqref{Holder-1} and \eqref{Holder-2} verify  the $1/\delta$- H\"older continuity of $\psi_i$.
$\Box$

\begin{remark} {\rm We note that the initial point $\psi_i(0)$ is the head of $E_i$,
 and  the terminate point $\psi_i(h_i)$ is the tail of $E_i$.
 }
\end{remark}

\begin{remark} {\rm In some text books, the space-filling curves is concerned;
for example, \emph{`Real Analysis'} by Stein and Shakarchi (2005),
\emph{`Topology'} by Munkres (2000), \emph{`Basic Topology'} by  Armstrong (1997). The above proof extends the arguments in these books.}
\end{remark}

\end{document}